\pdfoutput=1
\RequirePackage{ifpdf}
\ifpdf 
\documentclass[pdftex]{sigma}
\else
\documentclass{sigma}
\fi

\numberwithin{equation}{section}

\newtheorem{Theorem}{Theorem}[section]
\newtheorem{Corollary}[Theorem]{Corollary}
\newtheorem{Lemma}[Theorem]{Lemma}
\newtheorem{Proposition}[Theorem]{Proposition}
 { \theoremstyle{definition}
\newtheorem{Definition}[Theorem]{Definition}
\newtheorem{Example}[Theorem]{Example} }

\begin{document}

\allowdisplaybreaks

\newcommand{\arXivNumber}{1509.00175}

\renewcommand{\PaperNumber}{061}

\FirstPageHeading

\ShortArticleName{Geometric Monodromy around the Tropical Limit}

\ArticleName{Geometric Monodromy around the Tropical Limit}

\Author{Yuto YAMAMOTO}

\AuthorNameForHeading{Y.~Yamamoto}

\Address{Graduate School of Mathematical Sciences, The University of Tokyo,\\ 3-8-1 Komaba, Meguro, Tokyo, 153-8914, Japan}
\Email{\href{mailto:yuto@ms.u-tokyo.ac.jp}{yuto@ms.u-tokyo.ac.jp}}

\ArticleDates{Received September 02, 2015, in f\/inal form June 17, 2016; Published online June 24, 2016}

\Abstract{Let $\{V_q\}_{q}$ be a complex one-parameter family of smooth hypersurfaces in a~toric variety. In this paper, we give a concrete description of the monodromy transformation of~$\{V_q\}_q$ around $q=\infty$ in terms of tropical geometry. The main tool is the tropical locali\-zation introduced by Mikhalkin.}


\Classification{14T05; 14D05}

\section{Introduction}

Let $K := {\mathbb C}\{t\}$ be the convergent Laurent series f\/ield, equipped with the standard non-archi\-me\-dean valuation,
\begin{gather}\label{eq:val}
\operatorname{val} \colon \ K \longrightarrow {\mathbb Z} \cup \{ - \infty \},\qquad k=\sum_{j \in {\mathbb Z}}c_j t^j \mapsto - \min \{ j \in {\mathbb Z} \,|\, c_j \ne 0 \}.
\end{gather}
Let $n \in {\mathbb N}$ be a natural number and $M$ be a free ${\mathbb Z}$-module of rank $n+1$. We write $M_{\mathbb R}:=M \otimes_{\mathbb Z} {\mathbb R}$.
Let further $\Delta \subset M_{\mathbb R}$ be a convex lattice polytope, i.e., the convex hull of a f\/inite subset of~$M$. We set $A:= \Delta \cap M$.
Let $F=\sum\limits_{m \in A} k_m x^m \in K\big[x^\pm_{1},\dots,x^\pm_{n+1}\big]$ be a~Laurent polynomial over $K$ in $n+1$ variables such that $k_m \neq 0$ for all $m \in A$. We f\/ix a suf\/f\/iciently large $R \in {\mathbb R}^{>0}$ such that $1/R$ is smaller than the radius of convergence of $k_m$ for all $m \in A$, and set $S_R^1 := \{ z \in {\mathbb C} \,|\, |z|=R \}$. For $q \in S_R^1$, let $f_q \in {\mathbb C}[x^\pm_{1},\dots,x^\pm_{n+1}]$ denote the polynomial obtained by substituting $1/q$ to $t$ in~$F$. Let ${\mathcal F}$ be the normal fan to $\Delta$ and ${\mathcal F}'$ be a unimodular subdivision of ${\mathcal F}$. Let $X_{{\mathcal F}'}({\mathbb C})$ denote the toric manifold over ${\mathbb C}$ associated with ${\mathcal F}'$. For each $q \in S_R^1$, we def\/ine $V_q \subset X_{{\mathcal F}'}({\mathbb C})$ as the hypersurface def\/ined by~$f_q$ in $X_{{\mathcal F}'}({\mathbb C})$. In this paper, we discuss the monodromy transformation of $\{V_q\}_{q \in S_R^1}$ around $q=\infty$. The limit $q \to \infty$ is called the tropical limit in this paper. The motivation to address this problem comes from the calculation of monodromies of period maps.

Let $\operatorname{trop}(F) \colon {\mathbb R}^{n+1} \to {\mathbb R}$ be the tropicalization of $F$ def\/ined by
\begin{gather}\label{eq:trp}
\operatorname{trop}(F)(X_1, \dots , X_{n+1}):=\max_{m \in A} \big\{{\operatorname{val}}(k_m)+m_1X_1+ \cdots +m_{n+1}X_{n+1} \big\}.
\end{gather}
The non-dif\/ferentiable locus of $\operatorname{trop}(F)$ is called the tropical hypersurface def\/ined by $\operatorname{trop}(F)$
and denoted by $V(\operatorname{trop}(F))$. The tropical hypersurface $V(\operatorname{trop}(F))$ is a rational polyhedral complex of dimension $n$.
The main theorem of this paper is Theorem~\ref{th:main}, which gives a concrete description of the monodromy transformation of $\{V_q\}_{q \in S_R^1}$ in terms of the tropical hypersur\-fa\-ce~$V(\operatorname{trop}(F))$ in the case where~$V(\operatorname{trop}(F))$ is smooth (see Def\/inition~\ref{df:hypsm}). The monodromy of~$\{V_q\}_{q \in S_R^1}$ is also discussed in \cite[Appendix~B.2]{DKK} and Theorem~\ref{th:main} is covered by \cite[Proposition~B.17]{DKK}. However, this paper aims to make the relation of the monodromy of $\{V_q\}_{q \in S_R^1}$ to tropical geometry clear. We give a self-contained proof and explicit examples.

When $\Delta$ is smooth and ref\/lexive and the polynomial $F$ gives a central subdivision of $\Delta$, Zharkov~\cite{MR1738179} also gave a concrete description of the monodromy transformation of $\{V_q\}_{q \in S_R^1}$. The idea of his description is the same as that of ours. By treating his construction systematically, we generalize his result to the case where~$\Delta$ is any polytope and the subdivision of~$\Delta$ given by~$F$ is not necessarily central.

Since the claim of Theorem~\ref{th:main} is technical and it is necessary to make preparations in order to state it, we do not state it here and discuss its corollary in the following. Assume $n=1$. Let $\{\rho_i \}_{i \in \{1, \dots, d\}}$ be the set of all bounded edges of $V(\operatorname{trop}(F))$.
For each $\rho_i$, let $\nu_{i1}, \nu_{i2} \in {\mathbb R}^{n+1}$ be the endpoints of~$\rho_i$. Let further $V \in {\mathbb Z}^{n+1}$ be the primitive vector such that $\nu_{i1}-\nu_{i2}=lV$ for some $l \in {\mathbb R}^{>0}$. We def\/ine the length $L(\rho_i)$ of $\rho_i$ as $l \in {\mathbb R}^{>0}$. Assume that the tropical hypersur\-face~$V(\operatorname{trop}(F))$ is \textit{smooth}, in the sense that for any vertex $\nu$ of $V(\operatorname{trop}(F))$, there exists a~${\mathbb Z}$-af\/f\/ine transformation $\big((m_{ij})_{1 \leq i,j \leq 2}, (r_i)_{i=1,2} \big) \in {\rm GL}_2({\mathbb Z}) \ltimes {\mathbb R}^2$ such that in the coordinate $(Y_1, Y_2)$ on~${\mathbb R}^2$ def\/ined by
\begin{gather*}
Y_1=m_{11}X_1+m_{12}X_2+r_1, \qquad Y_2=m_{21}X_1+m_{22}X_2+r_2,
\end{gather*}
the tropical hypersurface $V(\operatorname{trop}(F))$ coincides locally with the tropical hyperplane def\/ined by $\max \{0, Y_1, Y_2 \}$ around $\nu$.
Then we have $\nu_{i1}, \nu_{i2} \in {\mathbb Z}^{n+1}$. The amoeba of $V_q$ converges to the tropical hypersurface $V(\operatorname{trop}(F))$ as $q \to \infty$ in the Hausdorf\/f metric~\cite{MR2079993, ISSN:1401-5617} and the hypersurface~$V_q$ is obtained by `thickening' the amoeba of~$V_q$. Let $C_i$ $(i=1, \dots, d)$ be the simple closed curve in $V_{q=R}$ turning around~$\rho_i$ (see Fig.~\ref{fg:intro} for an example).
Let further $T_{i} \colon V_R \to V_R$ be the Dehn twist along $C_i$.

\begin{Corollary}\label{cr:iwa} If $n=1$ and $V(\operatorname{trop}(F))$ is smooth, then the monodromy transformation of $\{V_q\}_{q \in S_R^1}$ around $q=\infty$ is given by $T_{1}^{L(\rho_1)} \circ \cdots \circ T_{d}^{L(\rho_d)}$.
\end{Corollary}

Corollary~\ref{cr:iwa} is conjectured by Iwao \cite{IwaoLecture2010}.
Let us illustrate this claim with a simple example.
Consider the polynomial $F$ given by
\begin{gather}\label{eq:hypellip}
F(x_1,x_2)=x_2^2+x_2 \big( x_1^3 + t^{-2} x_1^2 +t^{-2} x_1 + t^{-1} \big)+1.
\end{gather}
Then we have
\begin{gather*}
f_q(x_1,x_2) =x_2^2+x_2 \big( x_1^3 + q^{2} x_1^2 +q^{2} x_1 + q^{1} \big)+1, \\
\operatorname{trop}(F)(X_1,X_2) =\max \{ 2X_2 , 3X_1+X_2 , 2X_1+X_2+2 , X_1+X_2+2, X_2+1, 0 \}.
\end{gather*}
The tropical hypersurface $V(\operatorname{trop}(F))$ and the hypersurface $V_q$ in this case are shown in Fig.~\ref{fg:intro}. Let~$\rho_i$ and~$C_i$ $(i=1,\dots,7)$ denote edges of $V(\operatorname{trop}(F))$ and simple closed curves in $V_q$ as shown in Fig.~\ref{fg:intro}. Then the edges $\rho_1, \dots , \rho_7$ correspond to the simple closed curves $C_1, \dots, C_7$, respectively. By simple calculations, we have
\begin{gather*}
L(\rho_1)=2, \qquad\! L(\rho_2)=4, \qquad\! L(\rho_3)=12,\qquad\! L(\rho_4)=L(\rho_5)=1, \qquad\! L(\rho_6)=L(\rho_7)=2.
\end{gather*}
It follows from Corollary~\ref{cr:iwa} that the monodromy transformation of $\{V_q\}_{q \in S_R^1}$ is given by $T_1^2 \circ T_2^4 \circ T_3^{12} \circ T_4 \circ T_5 \circ T_6^2 \circ T_7^2$.

\begin{figure}[t]\centering
\includegraphics[scale=0.45]{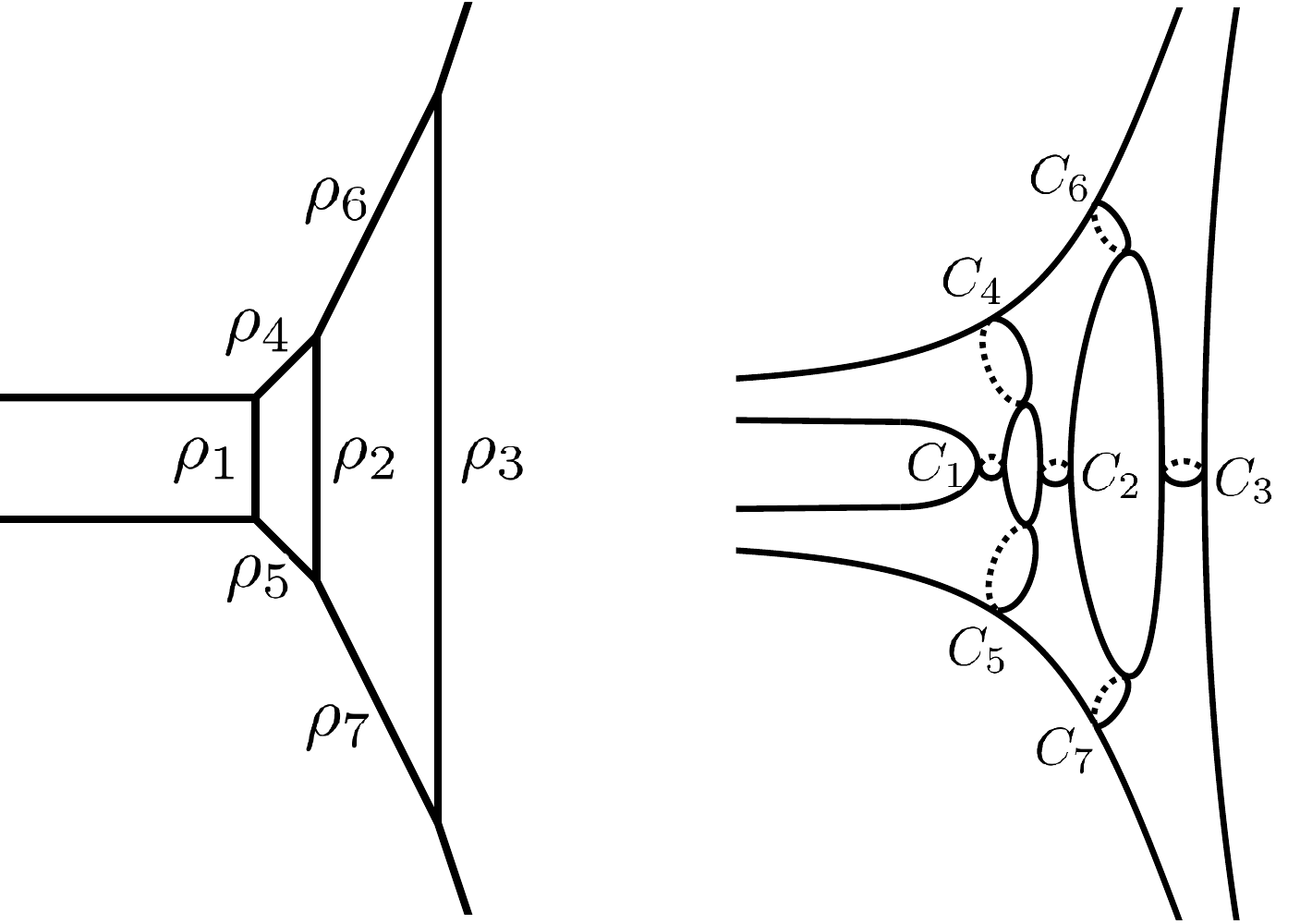}
\caption{The tropical hypersurface $V(\operatorname{trop}(F))$ and the hypersurface $V_q$ for~\eqref{eq:hypellip}.}\label{fg:intro}
\end{figure}

The organization of this paper is as follows: First, we set up the notation in Section~\ref{sc:2}. In Section~\ref{sc:3}, we recall the notion of the tropical localization introduced by Mikhalkin~\cite{MR2079993}. This is the main tool to construct the monodromy transformation of $\{V_q\}_{q \in S_R^1}$. In Section~\ref{sc:4}, we give an explicit description of the monodromy transformations in any dimension. In Section~\ref{sc:5}, we show that Corollary~\ref{cr:iwa} follows from Theorem~\ref{th:main}. In Section~\ref{sc:6}, we give examples in dimension~$1$ and~$2$. In Section~\ref{sc:7}, we discuss the relation between Zharkov's description and ours. This section may also be useful for understanding this paper and a possible f\/irst step for getting our idea.

\section{Preliminaries}\label{sc:2}

\subsection{Tropical toric varieties}

Let $M$ be a free ${\mathbb Z}$-module of rank $n+1$ and $N:=\operatorname{Hom}_{\mathbb Z} (M,{\mathbb Z})$ be the dual lattice of $M$.
We set $M_{\mathbb R}:=M \otimes_{\mathbb Z} {\mathbb R}$ and $N_{\mathbb R}:=N \otimes_{\mathbb Z} {\mathbb R}=\operatorname{Hom}_{\mathbb Z} (M,{\mathbb R})$. We have a canonical ${\mathbb R}$-bilinear pairing
\begin{gather*}
\langle - ,- \rangle \colon \ M_{\mathbb R} \times N_{\mathbb R} \to {\mathbb R}.
\end{gather*}
Let ${\mathcal F}$ be a fan in $N_{\mathbb R}$. We write the toric variety associated with ${\mathcal F}$ over ${\mathbb C}$ as $X_{\mathcal F}({\mathbb C})$.
For each cone $\sigma \in {\mathcal F}$, we set
\begin{gather*}
\sigma^\vee := \{ m \in M_{\mathbb R} \,|\, \langle m,n \rangle \geq 0 \mathrm{\ for\ all\ } n \in \sigma \},\\
\sigma^\perp := \{ m \in M_{\mathbb R} \,|\, \langle m,n \rangle =0 \mathrm{\ for\ all\ } n \in \sigma \}.
\end{gather*}
Let $U_\sigma({\mathbb C}):=\operatorname{Hom}(\sigma^\vee \cap M, {\mathbb C})$ denote the af\/f\/ine toric variety and $O_\sigma({\mathbb C}):=\operatorname{Hom} ( \sigma^{\perp} \cap M, {\mathbb C}^\ast )$ denote the torus orbit corresponding to~$\sigma$. We write the closure of $O_\sigma({\mathbb C})$ in $X_{\mathcal F}({\mathbb C})$ as $X_{{\mathcal F},\sigma}({\mathbb C})$.

Let ${\mathbb T}:={\mathbb R} \cup \{ -\infty \}$ be the tropical semi-ring, equipped with the following arithmetic operations for any $a, b \in {\mathbb T}$;
\begin{gather*}
a \oplus b :=\max \{ a, b \}, \qquad a \odot b :=a+b.
\end{gather*}
We can also def\/ine the toric variety over ${\mathbb T}$ as follows. For each cone $\sigma \in {\mathcal F}$, we def\/ine $U_\sigma({\mathbb T})$ as the set of monoid homomorphisms $\sigma^\vee \cap M \to ({\mathbb T}, \odot)$,
\begin{gather*}
U_\sigma({\mathbb T}):=\operatorname{Hom}\big(\sigma^\vee \cap M, {\mathbb T}\big)
\end{gather*}
with the compact open topology. For cones $\sigma, \tau \in {\mathcal F}$ such that $\sigma \prec \tau$, we have a natural immersion,
\begin{gather*}
U_\sigma({\mathbb T}) \to U_\tau({\mathbb T}),\qquad \big(v \colon \sigma^\vee \cap M \to {\mathbb T}\big) \mapsto \big(\tau^\vee \cap M \subset \sigma^\vee \cap M \xrightarrow{v} {\mathbb T}\big),
\end{gather*}
where $\sigma \prec \tau$ means that $\sigma$ is a face of $\tau$. By gluing $\{ U_\sigma({\mathbb T}) \}_{\sigma \in {\mathcal F}}$ with each other, we have the tropical toric variety $X_{\mathcal F}({\mathbb T})$ associated with~${\mathcal F}$,
\begin{gather*}
X_{\mathcal F}({\mathbb T}):=\bigg( \coprod_{\sigma \in {\mathcal F}} U_\sigma({\mathbb T}) \bigg) \bigg/ {\sim}.
\end{gather*}
Tropical toric varieties are f\/irst introduced by Kajiwara~\cite{KajiwaraPre}, see~\cite{KajiwaraPre} or~\cite{MR2428356} for details. For a~projective toric variety, the associated tropical toric variety is homeomorphic to the moment polytope of it \cite[Remark~1.3]{MR2428356}.

\begin{Example}The tropical projective space of $n$-dimension is homeomorphic to the $n$-dimen\-sio\-nal simplex.
\end{Example}

We def\/ine the torus orbit $O_\sigma({\mathbb T})$ over ${\mathbb T}$ corresponding to $\sigma$ by
\begin{gather*}
O_\sigma({\mathbb T}):=\operatorname{Hom}\big(\sigma^\perp \cap M,{\mathbb R}\big),
\end{gather*}
and write the closure of $O_\sigma({\mathbb T})$ in $X_{\mathcal F}({\mathbb T})$ as $X_{{\mathcal F},\sigma}({\mathbb T})$. Let $R \in {\mathbb R}^{>0}$ be a~positive real number and $\operatorname{Log}_R \colon {\mathbb C} \to {\mathbb T}$ denote the map def\/ined by
\begin{gather*}
c \mapsto \begin{cases}
\log_R|c|, & c \neq 0, \\
 -\infty, & c=0.
\end{cases}
\end{gather*}
We have a canonical map $\operatorname{Log}_R \colon X_{\mathcal F}({\mathbb C}) \to X_{\mathcal F}({\mathbb T})$ def\/ined by
\begin{gather}\label{eq:log}
U_\sigma({\mathbb C})=\operatorname{Hom}\big(\sigma^\vee \cap M, {\mathbb C}\big) \to U_\sigma({\mathbb T})=\operatorname{Hom}\big(\sigma^\vee \cap M, {\mathbb T}\big),\qquad v \mapsto \operatorname{Log}_R \circ v.
\end{gather}

\subsection{Polyhedral complex}

We def\/ine the product ${\mathbb R}^{\geq0} \times {\mathbb T} \to {\mathbb T}$ by
\begin{gather*}
r \cdot t := \begin{cases}
r \times t, & t \neq -\infty, \\
-\infty, & r \neq 0, \ t=-\infty, \\
0, & r=0, \ t=-\infty,
\end{cases}
\end{gather*}
for $r \in {\mathbb R}^{\geq 0}$ and $t \in {\mathbb T}$. Here, $\times$ denotes the ordinary multiplication of~${\mathbb R}$. We also def\/ine the product $\big({\mathbb R}^{\geq0}\big)^{n+1} \times {\mathbb T}^{n+1} \to {\mathbb T}$ by
\begin{gather*}
a \cdot b :=\sum_{i=1}^{n+1} a_i \cdot b_i,
\end{gather*}
for $a=(a_1, \dots, a_{n+1}) \in \big({\mathbb R}^{\geq0}\big)^{n+1}$ and $b=(b_1, \dots, b_{n+1}) \in {\mathbb T}^{n+1}$. For each subset $I \subset \{1, \dots,$ $n+1 \}$, we set
\begin{gather*}
{\mathbb T}_{I}^{n+1}:= \big\{ X \in {\mathbb T}^{n+1} \,|\, X_i = -\infty \ \mathrm{for\ any\ } i \in I \big\}.
\end{gather*}

\begin{Definition} A subset $\rho$ of ${\mathbb T}^{n+1}$ is a \textit{convex polyhedron} if there exist a f\/inite collection $\{H_j\}_{j \in J}$ of half-spaces of the form
\begin{gather*}
H_j=\big\{ X \in {\mathbb T}^{n+1} \,|\, c_j \cdot X \leq d_j \big\}, \qquad c_j \in \big({\mathbb R}^{\geq 0}\big)^{n+1}, \qquad d_j \in {\mathbb R},
\end{gather*}
and a subset $I \subset \{1, \dots, n+1 \}$ such that
\begin{gather*}
\rho = \cap_{j \in J} H_j \cap {\mathbb T}_I^{n+1}.
\end{gather*}
A subset $\mu$ of $\rho$ is a \textit{face} of $\rho$ if there exist subsets $J' \subset J$ and $I' \subset \{1, \dots, n+1 \}$ such that $I' \supset I$ and
\begin{gather*}
\mu=\big\{ X \in \rho \,|\,
c_j \cdot X = d_j \ {\rm for\ all}\ j \in J', \ X_i =-\infty \ {\rm for\ all}\ i \in I'\big\}.
\end{gather*}
We write $\mu \prec \rho$ when $\mu$ is a face of $\rho$.
\end{Definition}

\begin{Definition}
A fan ${\mathcal F}$ in $N_{\mathbb R}$ is called {\it unimodular} if every cone in~${\mathcal F}$ can be generated by a~subset of a basis for~$N$.
\end{Definition}

Let ${\mathcal F}$ be a complete and unimodular fan in $N_{\mathbb R}$ in the following.

\begin{Definition} A subset $\rho$ of $X_{\mathcal F}({\mathbb T})$ is a \textit{convex polyhedron} $\rho$ if $\rho \cap U_\sigma({\mathbb T})$ is a convex polyhedron in $U_\sigma({\mathbb T}) \cong {\mathbb T}^{n+1}$ for any $(n+1)$-dimensional cone $\sigma \in {\mathcal F}$. A~subset $\mu$ in $\rho$ is called a~\textit{face} of~$\rho$ when $\mu \cap U_\sigma({\mathbb T})$ is a face of $\rho \cap U_\sigma({\mathbb T})$ for any $(n+1)$-dimensional cone $\sigma \in {\mathcal F}$. We write $\mu \prec \rho$ when~$\mu$ is a face of~$\rho$.
\end{Definition}

\begin{Definition}\label{df:cpx}
A f\/inite set $P$ of convex polyhedra in $X_{\mathcal F}({\mathbb T})$ is a \textit{polyhedral complex} if it satisf\/ies the following conditions:
\begin{itemize}\itemsep=0pt
\item For any convex polyhedron $\rho \in P$, all faces of $\rho$ are elements of $P$.
\item For any two convex polyhedra $\rho_1, \rho_2 \in P$, $\rho_1 \cap \rho_2$ is a face of $\rho_1$ and $\rho_2$.
\end{itemize}
Each element $\rho \in P$ is called a \textit{cell}. In particular, we call $\rho$ a \textit{$k$-cell} when~$\rho$ is $k$-dimensional.
\end{Definition}

Let $P$ be a polyhedral complex in $X_{{\mathcal F}}({\mathbb T})$. For each $\sigma \in {\mathcal F}$, we def\/ine
\begin{gather*}
P_\sigma:=\{ \rho \in P \,|\, \operatorname{relint}(\rho) \subset O_\sigma({\mathbb T}) \},
\end{gather*}
where $\operatorname{relint}(\rho)$ denotes the relative interior of $\rho$.

\subsection{Hypersurfaces in toric varieties}\label{sc:2.3}

Let $K := {\mathbb C}\{t\}$ be the convergent Laurent series f\/ield, equipped with the standard non-archi\-me\-dean valuation~\eqref{eq:val}. Let further $\Delta \subset M_{\mathbb R}$ be a convex lattice polytope. We set $A:=\Delta \cap M$. Let $F=\sum\limits_{m \in A} k_m x^m \in K \big[ x^\pm_{1},\dots,x^\pm_{n+1} \big] $ be a Laurent polynomial over~$K$ in~$n+1$ variables such that $k_m \neq 0$ for all $m \in A$. Let~${\mathcal F}$ denote the normal fan to~$\Delta$. We choose a unimodular subdivision~${\mathcal F}'$ of~${\mathcal F}$.

The tropicalization of $F$ is the piecewise-linear map $\operatorname{trop}(F) \colon O_{\{0\}}({\mathbb T}) \cong {\mathbb R}^{n+1} \to {\mathbb R}$ given by~\eqref{eq:trp}. Let $V_{\{0\}}(\operatorname{trop}(F))$ denote the non-dif\/ferentiable locus of $\operatorname{trop}(F)$ in $O_{\{0\}}({\mathbb T}) \cong {\mathbb R}^{n+1}$. Let further $V(\operatorname{trop}(F))$ denote the closure of $V_{\{0\}}(\operatorname{trop}(F))$ in $X_{{\mathcal F}'}({\mathbb T})$. The tropical hypersur\-face~$V(\operatorname{trop}(F))$ has a structure of a polyhedral complex in $X_{{\mathcal F}'}({\mathbb T})$. Let $P$ denote the polyhedral complex given by $V(\operatorname{trop}(F))$ in the following.

\begin{Example}\looseness=-1 Let $F$, $G$ be polynomials def\/ined by $F=1+x_1+x_2$ and $G=1+x_1+x_2+x_3$. Then the tropicalizations of $F$ and $G$ are $\operatorname{trop}(F)=\max \{ 0, X_1, X_2 \}$ and $\operatorname{trop}(G)=\max \{0, X_1, X_2, X_3\}$. The tropical hypersurfaces $V(\operatorname{trop}(F))$ and $V(\operatorname{trop}(G))$ are shown in Fig.~\ref{fg:troplane}. The polyhedral complex given by $V(\operatorname{trop}(F))$ consists of four $0$-cells and three $1$-cells. The polyhedral complex given by $V(\operatorname{trop}(G))$ consists of eleven $0$-cells, sixteen $1$-cells, and six $2$-cells.
\end{Example}

\begin{figure}[t]\centering
\includegraphics[scale=0.3]{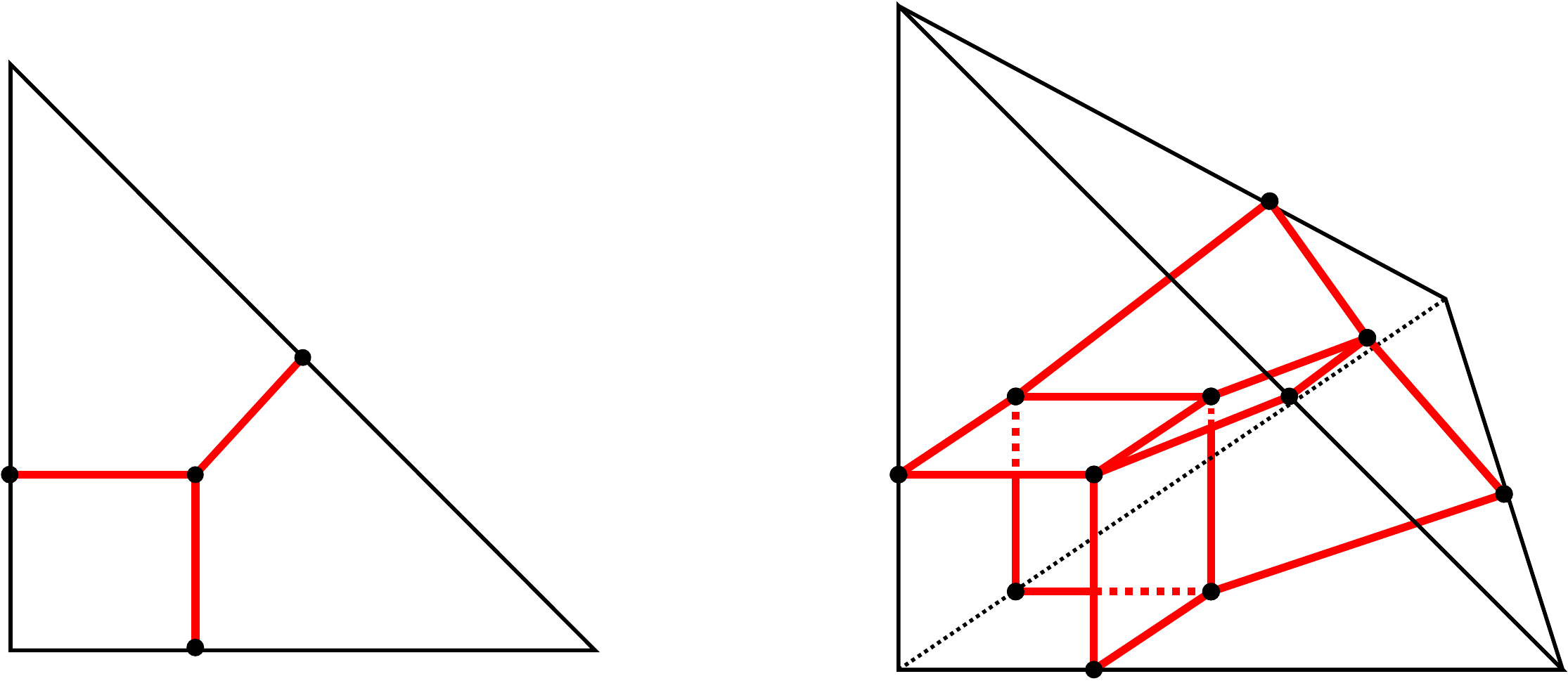}
\caption{Tropical hyperplanes of dimensions $1$ and $2$ in tropical projective spaces.}\label{fg:troplane}
\end{figure}

Let $v \colon A \to {\mathbb Z}$ be the function def\/ined by $v(m):=\operatorname{val}(k_m)$. Let further $\Gamma_v$ be the subset in $M_{\mathbb R} \times {\mathbb R}$ def\/ined by
\begin{gather*}
\Gamma_v := \{ (m,r) \in A \times {\mathbb R} \,|\, r \leq v(m) \},
\end{gather*}
and $\operatorname{conv}(\Gamma_v)$ be the convex hull of $\Gamma_v$ in $M_{\mathbb R} \times {\mathbb R}$. We write the polyhedral subdivision of~$\Delta$ given by the projections of all bounded faces of $\operatorname{conv}(\Gamma_v)$ to $M_{{\mathbb R}}$ as ${\mathcal D}_v$. Note that all vertices of any polyhedron in ${\mathcal D}_v$ are contained in~$M$. It is well known that the tropical hypersur\-face~$V_{\{0\}}(\operatorname{trop}(F))$ is dual to the polyhedral subdivision ${\mathcal D}_v$ \cite[Proposition~3.1.6]{MR3287221}.

\begin{Definition}\label{df:hypsm} The polyhedral subdivision ${\mathcal D}_v$ is \textit{unimodular} if all elements of ${\mathcal D}_v$ are simplices of volume $\frac{1}{(n+1)!}$. We say~$V(\operatorname{trop}(F))$ is \textit{smooth} in this case.
\end{Definition}

\begin{Example} Consider the polynomial $F=t^{-1}+x_1+x_2+x_1^{-1}x_2^{-1}$. In this case, the func\-tion~$v$ is given by $v((0,0))=1$ and $v((1,0))=v((0,1))=v((-1,-1))=0$. The set $\operatorname{conv}(\Gamma_v)$ and the polyhedral subdivision ${\mathcal D}_v$ are shown in Fig.~\ref{fg:latdiv}.
The polyhedral subdivision ${\mathcal D}_v$ is unimodular and $V(\operatorname{trop}(F))$ is smooth in this case.
\end{Example}

\begin{figure}[t]\centering
\includegraphics[scale=0.6]{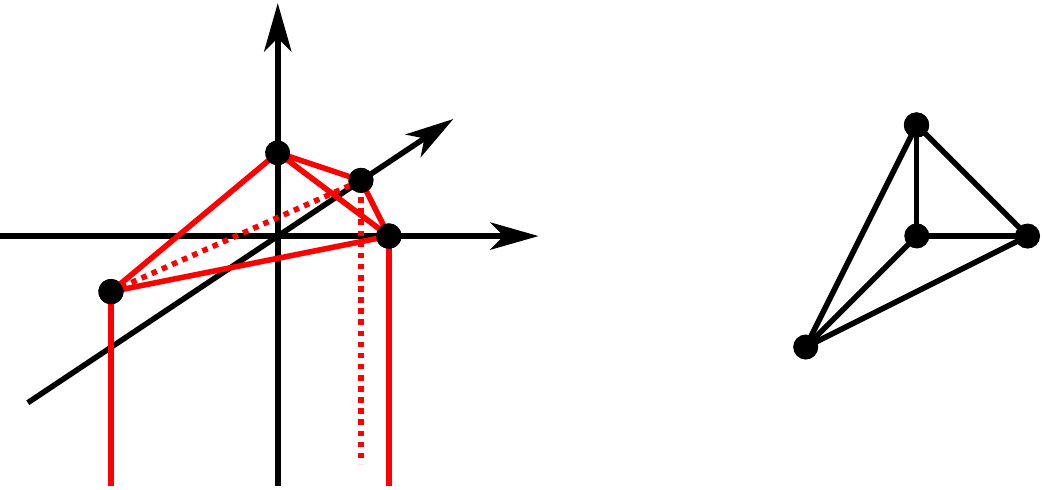}
\caption{The set $\operatorname{conv}(\Gamma_v)$ and the polyhedral subdivision ${\mathcal D}_v$ for $F=t^{-1}+x_1+x_2+x_1^{-1}x_2^{-1}$.}\label{fg:latdiv}
\end{figure}

We set $v_m:=\operatorname{val}(k_m)$ for $m \in A$. For each $\mu \in P_{\{0\}}$, we def\/ine the subset $A_\mu \subset A$ as the set of elements of~$A$ to which the dominant terms of~$F$ at~$\mu$ corresponds:
\begin{gather}\label{eq:A}
A_\mu := \big\{ m \in A \,|\, v_m +m \cdot X = \operatorname{trop}(F)(X) \mathrm{\ for\ all\ } X \in \mu \cap O_{\{0\}}({\mathbb T}) \big\}.
\end{gather}

\begin{Lemma}[\protect{\cite[Lemma 6.5]{MR2079993}}] Assume that the dimension of $\mu \in P_{\{0\}}$ is~$k$ $(0 \leq k \leq n)$. If the tropical hypersurface $V(\operatorname{trop}(F))$ is smooth, then the number of elements of $A_\mu$ is $n+2-k$.
\end{Lemma}

Assume that $V(\operatorname{trop}(F))$ is smooth. We f\/ix a suf\/f\/iciently large $R \in {\mathbb R}^{>0}$ such that $1/R$ is smaller than the radius of convergence of $k_m$ for all $m \in A$, and set $S_R^1 := \{ z \in {\mathbb C} \,|\, |z|=R \}$. For $q \in S_R^1$, let $f_q \in {\mathbb C}\big[x^\pm_{1}, \dots,x^\pm_{n+1}\big]$ be the Laurent polynomial obtained by substituting~$1/q$ to~$t$ in~$F$. We write the closure of $\big\{ x \in O_{\{0\}}({\mathbb C}) \,|\, f_q(x)=0 \big\}$ in $X_{{\mathcal F}'}({\mathbb C})$ as~$V_q$.

Let $\sigma \in {\mathcal F}'$ be an $l$-dimensional cone. For $\mu \in P_\sigma$, let $\mu' \in P_{\{0\}}$ be the cell such that $\mu = \mu' \cap X_{{\mathcal F}',\sigma}$. We assume that the dimension of~$\mu'$ is~$k$. Here, we have $l \leq k$. We def\/ine \textit{standard coordinates} on $O_\sigma({\mathbb C})$ and $O_\sigma({\mathbb T})$ with respect to~$\mu$ as follows. First, we number all elements of~$A_{\mu'}$ from~$0$ to $n+1-k$ and write them as $(m_0, \dots, m_{n+1-k})$. We set
\begin{gather*}
\tilde{x}_i:=q^{v_{m_i}}x^{m_i}/q^{v_{m_0}}x^{m_0},\qquad \widetilde{X}_i :=(v_{m_{i}}+{m_i} \cdot X)-(v_{m_{0}}+{m_0} \cdot X),
\end{gather*}
for $i=1,\dots, n+1-k$. Since $V(\operatorname{trop}(F))$ is smooth, we can extend $(\tilde{x}_1, \dots, \tilde{x}_{n+1-k})$ and $(\widetilde{X}_1, \dots, \widetilde{X}_{n+1-k})$ to $(\tilde{x}_1, \dots, \tilde{x}_{n+1-l})$ and $(\widetilde{X}_1, \dots, \widetilde{X}_{n+1-l})$ which form coordinate systems on $O_\sigma({\mathbb C})$ and~$O_\sigma({\mathbb T})$ respectively by setting
\begin{gather*}
\tilde{x}_i :=q^{a_i} \prod_{j=1}^{n+1}x_j^{b_{ij}},\qquad \widetilde{X}_i :=a_i+\sum_{j=1}^{n+1} b_{ij}X_j,
\end{gather*}
for $i=n+2-k,\dots, n+1-l$. Here, numbers $a_i$ and~$b_{ij}$ are appropriate integral numbers. We call $(\tilde{x}_1, \dots, \tilde{x}_{n+1-l})$ and $(\widetilde{X}_1, \dots, \widetilde{X}_{n+1-l})$ \textit{standard coordinates with respect to}~$\mu$. There are some ambiguities of them resulting from dif\/ferent numbering of $(m_0, \dots, m_{n+1-k})$ and dif\/ferent choices of numbers~$a_i$ and~$b_{ij}$.

Let $H_\mu \colon ({\mathbb C}^{\ast})^{n+1-l} \to ({\mathbb C}^{\ast})^{n+1-l}$ be the map def\/ined by
\begin{gather*}
(x_1, \dots, x_{n+1-l}) \mapsto (\tilde{x}_1, \dots, \tilde{x}_{n+1-l}),
\end{gather*}
and $M_\mu \colon {\mathbb R}^{n+1-l} \to {\mathbb R}^{n+1-l}$ be the map def\/ined by
\begin{gather*}
(X_1, \dots, X_{n+1-l}) \mapsto \big(\widetilde{X}_1, \dots, \widetilde{X}_{n+1-l}\big).
\end{gather*}
Then the following diagram is commutative.
\begin{gather*}
\begin{CD}
({\mathbb C}^{\ast})^{n+1-l} @>H_\mu>> ({\mathbb C}^{\ast})^{n+1-l} \\
@V\operatorname{Log}_RVV @VV\operatorname{Log}_RV \\
{\mathbb R}^{n+1-l} @>M_\mu>> {\mathbb R}^{n+1-l},
\end{CD}
\end{gather*}
where the map $\operatorname{Log}_R \colon ({\mathbb C}^{\ast})^{n+1-l} \to {\mathbb R}^{n+1-l}$ is def\/ined by
\begin{gather*}
(x_1, \dots , x_{n+1-l}) \to (\log_R|x_1|, \dots ,\log_R|x_{n+1-l}|).
\end{gather*}

\begin{Example}
Let us consider the polynomial $F=t^{-1}+x_1+x_2+x_1^{-1}x_2^{-1}$ again. We have $f_q=q+x_1+x_2+x_1^{-1}x_2^{-1}$. The tropicalization of $F$ is $\operatorname{trop}(F)=\max \{ 1, X_1, X_2, -X_1-X_2 \}$. The tropical hypersurface def\/ined by $\operatorname{trop}(F)$ is shown in Fig.~\ref{fg:trophypex}. Let $\nu$ and $\mu$ denote the vertex and the edge of $V(\operatorname{trop}(F))$ as shown in Fig.~\ref{fg:trophypex}. The set $A_\nu$ is given by $\{(0,0), (1,0), (-1,-1) \}$. We set $y_1:=x_1/q=q^{-1}x_1$, $y_2:=x_1^{-1}x_2^{-1}/q=q^{-1}x_1^{-1}x_2^{-1}$ and $Y_1:=-1+X_1$, $Y_2:=-1-X_1-X_2$. Then the sets of function $(y_1, y_2)$ and $(Y_1, Y_2)$ form standard coordinates with respect to~$\nu$ on $O_{\{0\}}({\mathbb C})$ and $O_{\{0\}}({\mathbb T})$, respectively. The set $A_\mu$ is given by $\{(1,0), (-1,-1) \}$. We set $z_1:=x_1^{-1}x_2^{-1}/x_1=x_1^{-2}x_2^{-1}$ and $Z_1:=-2X_1-X_2$. For instance, if we set $z_2:=q^2x_1$ and $Z_2:=2+X_1$, then we have
\begin{gather*}
\det \begin{pmatrix}
-2 & 1 \\
-1 & 0
\end{pmatrix} =1.
\end{gather*}
Hence, the sets of functions $(z_1, z_2)$ and $(Z_1, Z_2)$ form standard coordinates with respect to~$\mu$ on~$O_{\{0\}}({\mathbb C})$ and $O_{\{0\}}({\mathbb T})$, respectively.
\end{Example}

\begin{figure}[t]\centering
\includegraphics[scale=0.45]{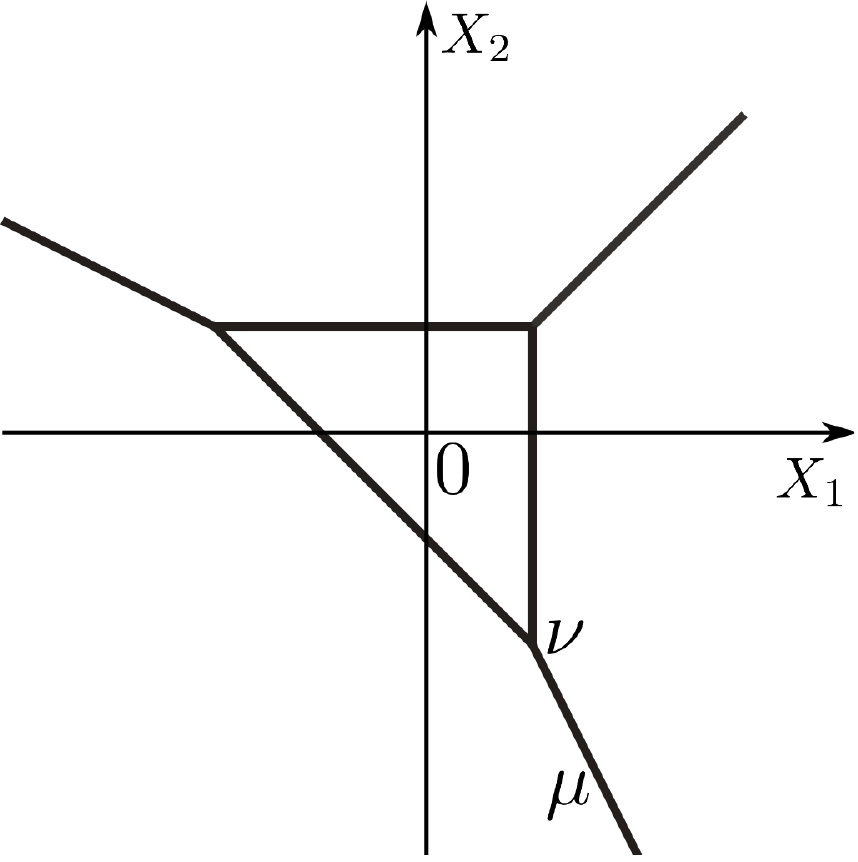}
\caption{The tropical hypersurface def\/ined by $\operatorname{trop}(F)=\max \{ 1, X_1, X_2 , -X_1-X_2 \}$.}\label{fg:trophypex}
\end{figure}

\section{Tropical localization}\label{sc:3}

Tropical localization is a way to simplify algebraic hypersurfaces around the tropical limit points by ignoring terms which are not dominant in the tropical limit. This technique is f\/irst introduced by Mikhalkin~\cite{MR2079993}. In this section, we give a concrete def\/ining function realizing the tropical localization based on the idea of Mikhalkin. There is also a similar construction of the tropical localization in~\cite{MR2240909}.

Let $K := {\mathbb C}\{t\}$ be the convergent Laurent series f\/ield, equipped with the standard non-archimedean valuation~\eqref{eq:val}. Let further $\Delta \subset M_{\mathbb R}$ be a convex lattice polytope. We set $A:=\Delta \cap M$. Let $F=\sum\limits_{m \in A} k_m x^m \in K \big[ x^\pm_{1},\dots,x^\pm_{n+1} \big] $ be a polynomial over $K$ such that $k_m \neq 0$ for all $m \in A$. We set $v_m:=\operatorname{val}(k_m)$. We f\/ix a suf\/f\/iciently large $R \in {\mathbb R}^{>0}$ such that $1/R$ is smaller than the radius of convergence of~$k_m$ for all $m \in A$, and set $S_R^1 := \{z \in {\mathbb C} \,|\, |z|=R \}$.
For $q \in S_R^1$, let $f_q \in {\mathbb C}\big[x^\pm_{1},\dots,x^\pm_{n+1}\big]$ denote the polynomial obtained by substituting~$1/q$ to~$t$ in~$F$.
Let~${\mathcal F}$ denote the normal fan to~$\Delta$. We choose a unimodular subdivision~${\mathcal F}'$ of~${\mathcal F}$. Let~$V_q$ be the hypersurface in~$X_{{\mathcal F}'}({\mathbb C})$ def\/ined by~$f_q$. Let further $V(\operatorname{trop}(F))$ be the tropical hypersurface in~$X_{{\mathcal F}'}({\mathbb C})$ def\/ined by~$\operatorname{trop}(F)$ and $P$ be the polyhedral complex in~$X_{{\mathcal F}'}({\mathbb T})$ given by $V(\operatorname{trop}(F))$.
Assume that~$V(\operatorname{trop}(F))$ is smooth (see Def\/inition~\ref{df:hypsm}).

Let $C_0, C_1 \in {\mathbb R} $ be constants such that $0<C_1<C_0 \ll 1$. Let $b \colon {\mathbb R} \to {\mathbb R}$ be a monotone $C^{\infty}$ function on~${\mathbb R}$ satisfying following conditions:
\begin{enumerate}\itemsep=0pt
\item[1)] $b(X)=1$ if and only if $X \leq C_1$,
\item[2)] $b(X)=0$ if and only if $X \geq C_0$.
\end{enumerate}
The graph of the function $b$ is shown in Fig.~\ref{fg:cutoff}.

\begin{figure}[t]\centering
\includegraphics[scale=0.4]{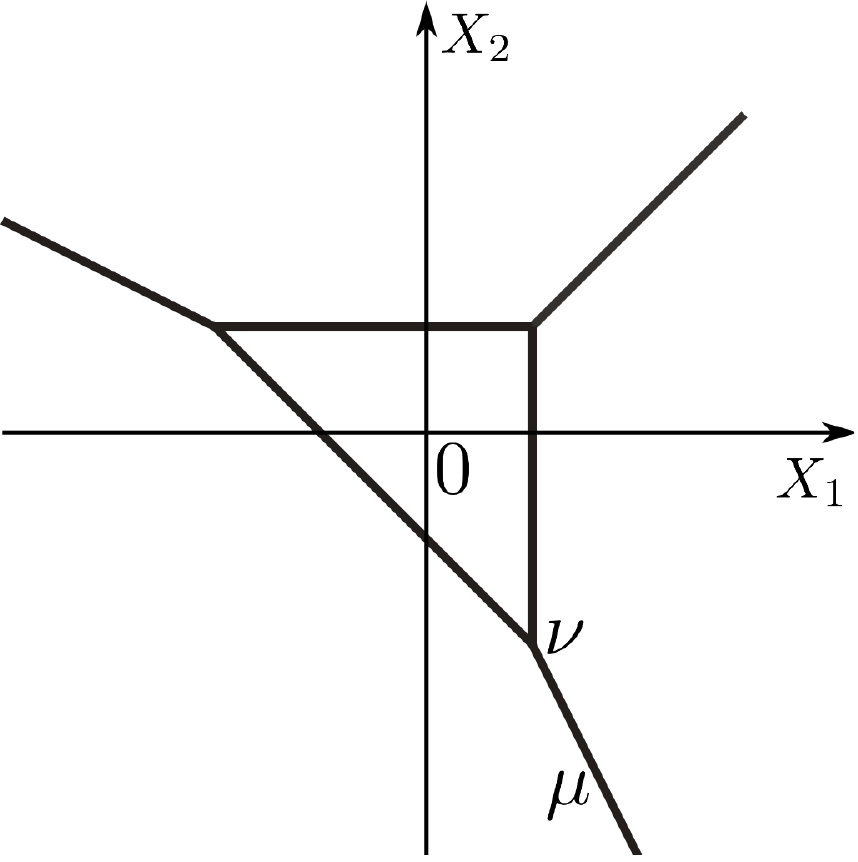}
\caption{The graph of the function~$b$.}\label{fg:cutoff}
\end{figure}

We def\/ine the tropical localization of the hypersurface $V_q$ as follows.

\begin{Definition}\label{df:loc} For each $m \in A$, let $b_m \colon O_{\{0\}}({\mathbb C}) \to {\mathbb R}$ be the function def\/ined by
\begin{gather*}
b_m(x) :=\prod_{i \in A}b\big(\log_R\big|q^{v_i}x^i\big|-\log_R\big|{q^{v_m}x^m}\big|\big).
\end{gather*}
In addition, let $\tilde{f}_q \colon O_{\{0\}}({\mathbb C}) \to {\mathbb C}$ be the function def\/ined by
\begin{gather*}
\tilde{f}_q(x) :=\sum_{m \in A} b_m(x)q^{v_m} x^m.
\end{gather*}
We def\/ine the {\it tropical localization} $W_q$ of~$V_q$ as the closure of $\{x \in O_{\{0\}}({\mathbb C}) \,|\, \tilde{f}_q(x)=0 \}$ in~$X_{{\mathcal F}'}({\mathbb C})$.
\end{Definition}

By applying Def\/inition~\ref{df:loc} to $f(x_1, \dots, x_{n+1})=1+x_1+ \dots +x_{n+1}$, we can construct the tropically localized hyperplane.

\begin{Definition}\label{df:lochyp}
We def\/ine the function $\tilde{f} \colon O_{\{0\}}({\mathbb C}) \to {\mathbb C}$ by
\begin{gather*}
\tilde{f}(x_1,\dots, x_{n+1}) :=\prod_{i=1}^{n+1}b(\log_R|x_i|)+\sum_{i=1}^{n+1}\left\{b(-\log_R|x_i|) \prod_{j=1}^{n+1}b(\log_R|x_j|-\log_R|x_i|) \right\}x_i,
\end{gather*}
We call the submanifold def\/ined as the zero locus of $\tilde{f}$ the {\it tropically localized hyperplane}.
\end{Definition}

\begin{Definition}\label{df:nbd} For each $\mu \in P$, we def\/ine $D_\mu \subset X_{{\mathcal F}'}({\mathbb C})$ and $\widehat{D}_\mu \subset X_{{\mathcal F}'}({\mathbb T})$ as follows. For $\mu \in P_{\{0\}}$, we def\/ine $D_\mu \subset X_{{\mathcal F}'}({\mathbb C})$ and $\widehat{D}_\mu \subset X_{{\mathcal F}'}({\mathbb T})$ by
\begin{gather*}
D_\mu:=
\overline{
\left\{ x \in O_{\{0\}}({\mathbb C}) \left|\,
\begin{array}{@{}l@{}}
b_{m}(x)>0\ {\rm for\ } m \in A_\mu, \\
b_{m}(x)=0\ {\rm for\ } m \in A \setminus A_\mu\\
\end{array}
\right.\right\}
},\\
\widehat{D}_\mu:=
\overline{
\left\{ X \in O_{\{0\}}({\mathbb T}) \left|\,
\begin{array}{@{}l@{}}
|(v_{m'}+m' \cdot X)-(v_{m}+m \cdot X)| < C_0\ {\rm for\ } m, m' \in A_\mu,\\
{\rm for\ any\ } m \in A \setminus A_\mu, {\rm \ there\ exists\ } m' \in A_\mu \\
{\rm such\ that\ } (v_{m'}+m' \cdot X)-(v_{m}+m \cdot X) \geq C_0 \\
\end{array}
\right.\right\}},
\end{gather*}
where $A_\mu \subset A$ is the set def\/ined in~\eqref{eq:A} and the overlines mean the closure in $X_{{\mathcal F}'}({\mathbb C})$ and~$X_{{\mathcal F}'}({\mathbb T})$, respectively.

For $\mu \in P_\sigma$ $(\sigma \neq \{0\})$, let $\mu' \in P_{\{0\}}$ be the cell such that $\mu = \mu' \cap X_{{\mathcal F}',\sigma}({\mathbb T})$. We def\/ine $D_\mu \subset X_{{\mathcal F}',\sigma}({\mathbb C})$ and $\widehat{D}_\mu \subset X_{{\mathcal F}',\sigma}({\mathbb T})$ by
\begin{gather}\label{eq:D}
D_\mu :=D_{\mu'} \cap X_{{\mathcal F}',\sigma}({\mathbb C}), \qquad \widehat{D}_\mu :=\widehat{D}_{\mu'} \cap X_{{\mathcal F}',\sigma}({\mathbb T}).
\end{gather}
\end{Definition}

The monomial $v_{m}+m \cdot X\ (m \in A_\mu)$ of $\operatorname{trop}(F)$ corresponds to the monomial $k_{m}x^{m}$ of~$F$. Hence, we have $D_\mu=(\operatorname{Log}_R)^{-1}(\widehat{D}_\mu)$ for any $\mu \in P$, where $\operatorname{Log}_R$ is the map from $X_{{\mathcal F}'}({\mathbb C})$ to~$X_{{\mathcal F}'}({\mathbb T})$ def\/ined in~\eqref{eq:log}.

\begin{Example} Consider the polynomial $F=t^{-1}+x_1+x_2+x_1^{-1}x_2^{-1}$. The tropical hypersur\-face~$V(\operatorname{trop}(F))$ and the regions $\{\widehat{D}_\mu\}_{\mu \in P}$ for $F$ are shown in Figs.~\ref{fg:trophypex2} and~\ref{fg:trophypex3}.
$\nu_i$ and $\mu_i$ $(i=1, \dots, 6)$ denote vertices and edges of $V(\operatorname{trop}(F))$ respectively as shown in Fig.~\ref{fg:trophypex2}.
Each~$\widehat{D}_{\nu_i}$ is the region colored in dark gray and each~$\widehat{D}_{\mu_i}$ is the region colored in light gray as shown in Fig.~\ref{fg:trophypex3}.
\end{Example}

\begin{figure}[t]\centering
\begin{minipage}[b]{0.48\hsize}
\centering
\includegraphics[width=40mm]{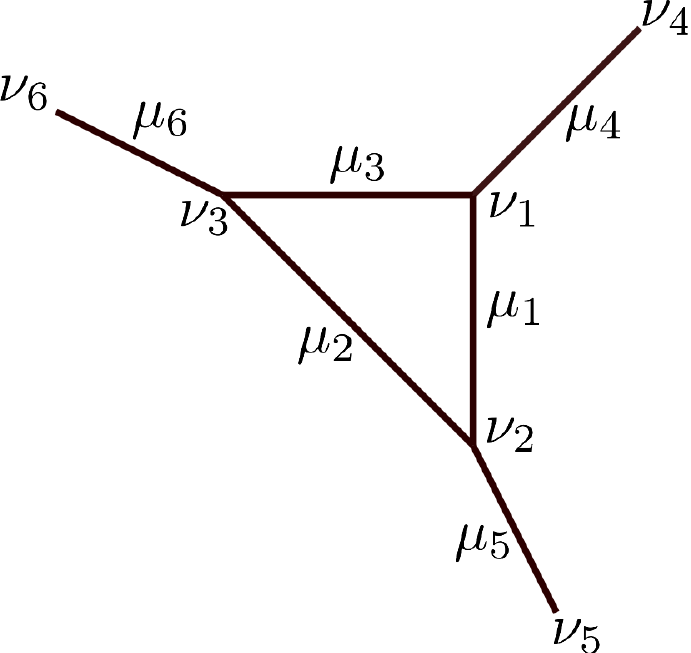}
\caption{The tropical hypersurface $V(\operatorname{trop}(F))$ for $F=t^{-1}+x_1+x_2+x_1^{-1}x_2^{-1}$.}
\label{fg:trophypex2}
\end{minipage}\quad
\begin{minipage}[b]{0.48\hsize}
\centering
\includegraphics[width=40mm]{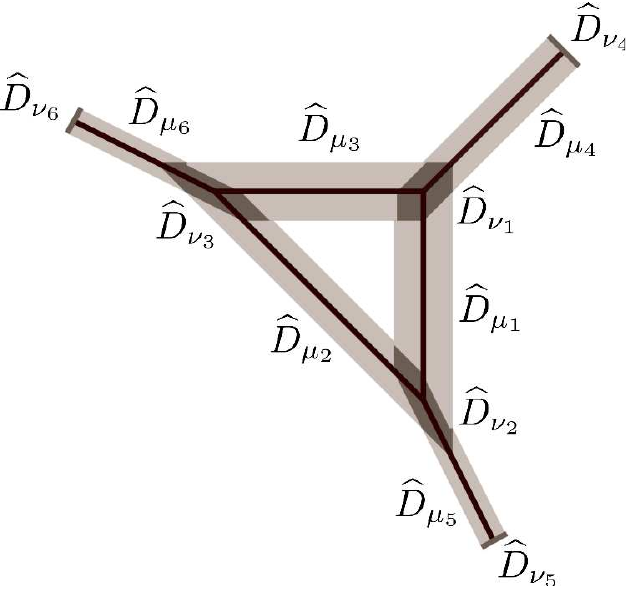}
\caption{The regions $\{\widehat{D}_{\rho}\}_\rho$ for $F=t^{-1}+x_1+x_2+x_1^{-1}x_2^{-1}$.}\label{fg:trophypex3}
\end{minipage}
\end{figure}

\begin{Lemma}\label{lm:a}
If $C_0$ is sufficiently small, then $D_\rho \cap X_{{\mathcal F}',\sigma}({\mathbb C}) \neq \varnothing$ if and only if $\rho \cap X_{{\mathcal F}',\sigma}({\mathbb T}) \neq \varnothing$ for any $\rho \in P_{\{0\}}$ and $\sigma \in {\mathcal F}'$.
\end{Lemma}
\begin{proof} Assume that $\rho \cap X_{{\mathcal F}',\sigma}({\mathbb T}) \neq \varnothing$. We set $\mu:=\rho \cap X_{{\mathcal F}',\sigma}({\mathbb T})$.
We show that $\widehat{D}_\mu =\widehat{D}_\rho \cap X_{{\mathcal F}',\sigma}({\mathbb T}) \neq \varnothing$. If $C_0$ is suf\/f\/iciently small, points in $\rho$ which are suf\/f\/iciently far from all faces of~$\rho$ in~$P_{\{0\}}$ are contained in~$\widehat{D}_\rho$. It follows that points in $\mu$ which are suf\/f\/iciently far from all faces of~$\mu$ are contained in~$\widehat{D}_\rho$, and hence in~$\widehat{D}_\mu$. Conversely, assume that $\rho \cap X_{{\mathcal F}',\sigma}({\mathbb T}) = \varnothing$. Since the region~$\widehat{D}_\rho$ has to be near to the cell $\rho$ if $C_0$ is suf\/f\/iciently small, we have $\widehat{D}_\rho \cap X_{{\mathcal F}',\sigma}({\mathbb T}) = \varnothing$.
\end{proof}

\begin{Lemma} If $C_0$ is sufficiently small, then one has
\begin{gather*}
\bigcup_{\rho \in P_{\{0\}}} D_\rho = \bigcup_{\sigma \in {\mathcal F}'} \bigg\{\bigcup_{\mu \in P_\sigma} (D_\mu \cap O_\sigma({\mathbb C})) \bigg\}.
\end{gather*}
\end{Lemma}

\begin{proof}
It is obvious that the right-hand side is contained in the left-hand side. We show that the left-hand side is contained in the right-hand side. Let $x$ be any point in $D_\rho$ $(\rho \in P_{\{0\}})$. There exists the unique cone $\sigma \in {\mathcal F}'$ such that $x \in O_\sigma({\mathbb C})$. Then, the point~$x$ is contained in $D_\rho \cap X_{{\mathcal F}',\sigma}({\mathbb C})$. From Lemma~\ref{lm:a}, we have $\rho \cap X_{{\mathcal F}',\sigma}({\mathbb T}) \neq \varnothing$. We set $\mu:=\rho \cap X_{{\mathcal F}',\sigma}({\mathbb T})$. Then we have $D_\rho \cap X_{{\mathcal F}',\sigma}({\mathbb C}) = D_\mu$ from~\eqref{eq:D}. Hence one has $x \in D_\mu \cap O_\sigma({\mathbb C})$.
\end{proof}

For each subset $\{m_0, \dots, m_p\} \subset A \ (p \in {\mathbb Z}_{\geq0})$, we def\/ine
\begin{gather*}
D_{m_0, \dots, m_p}:=
\overline{
\left\{ x \in O_{\{0\}}({\mathbb C}) \left|\,
\begin{array}{@{}l@{}}
b_{m_i}(x)>0\ {\rm for\ } i=0, \dots, p,\\
b_{m}(x)=0\ {\rm for\ } m \in A \setminus \{m_0, \dots, m_{p}\}\\
\end{array}
\right.\right\}
},
\end{gather*}
where the overline means the closure in $X_{{\mathcal F}'}({\mathbb C})$.

\begin{Lemma}\label{lm:oth}
Let $\{m_0, \dots, m_{p}\}$ be a subset of $A$ such that $p \geq 1$ and $\{m_0, \dots, m_{p}\} \neq A_\rho$ for any $\rho \in P_{\{0\}}$. If the constant $C_0$ is sufficiently small, then one has $D_{m_0, \dots, m_p}=\varnothing$.
\end{Lemma}
\begin{proof}
Let $H \subset O_{\{0\}}({\mathbb T})$ be the af\/f\/ine space def\/ined by \mbox{$v_{m_0}+m_0 \cdot X= \dots =v_{m_p}+m_p \cdot X$}.
First, we show that there exists a neighborhood $N$ of $H$ such that any of $v_{m_0}+m_0 \cdot X, \dots, v_{m_p}+m_p \cdot X$ do not coincide with $\operatorname{trop}(F)$ on $N$. Assume that there exists $X_0 \in H$ and $i \in \{1, \dots, p\}$ such that $v_{m_i}+m_i \cdot X_0 = \operatorname{trop}(F)(X_0)$. Then there exists $\rho' \in P_{\{0\}}$ such that \mbox{$X_0 \in \rho'$} and \mbox{$\{m_0, \dots, m_{p}\} \subset A_{\rho'}$}. Since $V(\operatorname{trop}(F))$ is smooth and locally coincides with the tropical hyperplane, there exists $\rho \in P_{\{0\}}$ such that $\rho' \prec \rho$ and $\{m_0, \dots, m_{p}\} = A_{\rho}$. This contradicts to the assumption. Hence, any of $v_{m_0}+m_0 \cdot X, \dots, v_{m_p}+m_p \cdot X$ do not coincide with $\operatorname{trop}(F)$ on~$H$. Then there exists a neighborhood $N$ of $H$ such that any of $v_{m_0}+m_0 \cdot X, \dots, v_{m_p}+m_p \cdot X$ do not coincide with $\operatorname{trop}(F)$ on~$N$.

Assume that $D_{m_0, \dots, m_p}$ is not empty. The dif\/ferences between the values of $v_{m_0}+m_0 \cdot X, \dots, v_{m_p}+m_p \cdot X$ are in the range of $\pm C_0$ on $\operatorname{Log}_R(D_{m_0, \dots, m_p}) \cap O_{\{0\}}({\mathbb T})$. Then the set $\operatorname{Log}_R(D_{m_0, \dots, m_p}) \cap O_{\{0\}}({\mathbb T})$ has to be in $N$ for a suf\/f\/iciently small constant $C_0$. The fact that any of $v_{m_0}+m_0 \cdot X, \dots, v_{m_p}+m_p \cdot X$ do not coincide with $\operatorname{trop}(F)$ on $N \supset \operatorname{Log}_R(D_{m_0, \dots, m_p}) \cap O_{\{0\}}({\mathbb T})$ contradicts to the def\/inition of $D_{m_0, \dots, m_p}$.
\end{proof}

\begin{Lemma}\label{lm:usig}
Let $\sigma \in {\mathcal F}'$ be a cone and $\mu_1, \mu_2 \in P_\sigma$ be cells.
Suppose that the constant $C_0$ is sufficiently small.
If $D_{\mu_1} \cap D_{\mu_2} \neq \varnothing$, then there exists $\mu \in P_\sigma$ such that $\mu \prec \mu_1, \mu_2$ and $D_{\mu_1} \cap D_{\mu_2} \subset D_\mu$.
\end{Lemma}
\begin{proof}
Let $\mu_1', \mu_2' \in P_{\{0\}}$ be the cells such that
$\mu_1=\mu_1' \cap X_{{\mathcal F}',\sigma}({\mathbb T})$ and $\mu_2=\mu_2' \cap X_{{\mathcal F}',\sigma}({\mathbb T})$.
We set $\{m_0, \dots, m_{p}\}:=A_{\mu_1'} \cup A_{\mu_2'}$.
Here, we have $D_{\mu_1'} \cap D_{\mu_2'} \subset D_{m_0, \dots, m_p}$.
If $D_{\mu_1} \cap D_{\mu_2} \neq \varnothing$, the set $D_{\mu_1'} \cap D_{\mu_2'} \supset D_{\mu_1} \cap D_{\mu_2}$ is also nonempty.
Hence, we have $D_{m_0, \dots, m_p} \neq \varnothing$.
From Lemma~\ref{lm:oth}, there must exists a cell $\rho \in P_{\{0\}}$ such that $A_\rho=\{m_0, \dots, m_p\}$ and $D_{m_0, \dots, m_p}=D_\rho$.
We have $D_{\mu_1'} \cap D_{\mu_2'} \subset D_\rho$ and $\rho \prec \mu_1', \mu_2'$.
Then the cell $\mu:=\rho \cap X_{{\mathcal F}',\sigma}({\mathbb T})$ satisf\/ies $D_{\mu_1} \cap D_{\mu_2} \subset D_\mu$ and $\mu \prec \mu_1, \mu_2$.
\end{proof}

The aim of this section is to prove the following theorem.

\begin{Theorem}\label{th:sub}
Fix a sufficiently small constant $C_0$.
For a sufficiently large $R \in {\mathbb R}^{>0}$, the tropical localization $W_q$ and the family of subsets $\{D_\mu\}_{\mu \in P}$ of $X_{{\mathcal F}'}({\mathbb C})$ satisfy the following conditions:
\begin{enumerate}\itemsep=0pt
\item[$1.$] For any $q \in S_R^1$, the submanifold $W_q$ is isotopic to $V_q$ in $X_{{\mathcal F}'}({\mathbb C})$.
\item[$2.$] For any $q \in S_R^1$, one has $W_q \subset \bigcup_{\rho \in P_{\{0\}}} D_\rho$.
\item[$3.$] Let $\sigma \in {\mathcal F}'$ be a cone and $\mu \in P_\sigma$ be a cell.
Let further $\mu' \in P_{\{0\}}$ be the cell such that $\mu = \mu' \cap X_{{\mathcal F}',\sigma}({\mathbb T})$.
Assume that the dimension of $\sigma$ and $\mu'$ is $l$ and $k$, respectively $(l \leq k)$.
Let $(\tilde{x}_1, \dots, \tilde{x}_{n+1-l})$ be a standard coordinate with respect to $\mu$ $($see Section~{\rm \ref{sc:2.3})}.
Then, the defining equation of $W_q$ on $D_\mu \cap O_\sigma({\mathbb C})$ coincides with that of the $(n-k)$-dimensional tropically localized hyperplane in $(\tilde{x}_1, \dots, \tilde{x}_{n+1-k})$ and is independent of the coordinate $(\tilde{x}_{n+2-k}, \dots, \tilde{x}_{n+1-l})$.
\end{enumerate}
\end{Theorem}

The outline of the proof of Theorem~\ref{th:sub} is as follows. In order to show the condition~1, we construct an isotopy $\{V_{q,s}\}_{s \in [0,1]}$ which connects $V_q$ and $W_q$. For $F=\sum\limits_{m \in A} k_m x^m$, we set $k_m=\sum\limits_{i \in {\mathbb Z}}c_{mi} t^i \in K$ $(c_{mi} \in {\mathbb C})$ and $c_m:=c_{m,-v_m}$. Let $d(s)$ be a real valued monotone $C^{\infty}$ function on ${\mathbb R}$ which has $1$ on $\{s \geq 2/3\}$ and $0$ on $\{s \leq 1/3\}$.
The graph of $d(s)$ is shown in Fig.~\ref{fg:cutoff2}.
For each $s \in [0,1]$, we def\/ine the functions $b_{m,s} \colon O_{\{0\}}({\mathbb C}) \to {\mathbb R}$ and $\tilde{f}_{q,s} \colon O_{\{0\}}({\mathbb C}) \to {\mathbb C}$ by
\begin{gather*}
b_{m,s}(x) :=(1-d(s))+d(s)b_m(x),\\
\tilde{f}_{q,s}(x) :=\sum_{m \in A} b_{m,s}(x){c_{m}}^{(1-d(s))}q^{v_m}x^m+(1-d(s)) \left\{ f_q-\sum_{m \in A}{c_{m}}q^{v_m}x^m \right\},
\end{gather*}
where the branch of $c_m^{(1-d(s))}$ is determined by $0 \leq \arg(c_m) < 2 \pi$.
Let $V_{q,s}$ be the closure of $\{x \in O_{\{0\}}({\mathbb C})\,|\, \tilde{f}_{q,s}(x)=0\}$ in $X_{{\mathcal F}'}({\mathbb C})$.
Then we have $\tilde{f}_{q,0}=f_q, \tilde{f}_{q,1}=\tilde{f}_q$ and $V_{q,0}=V_q$, $V_{q,1}=W_q$.

\begin{figure}[t]\centering
\includegraphics[scale=0.4]{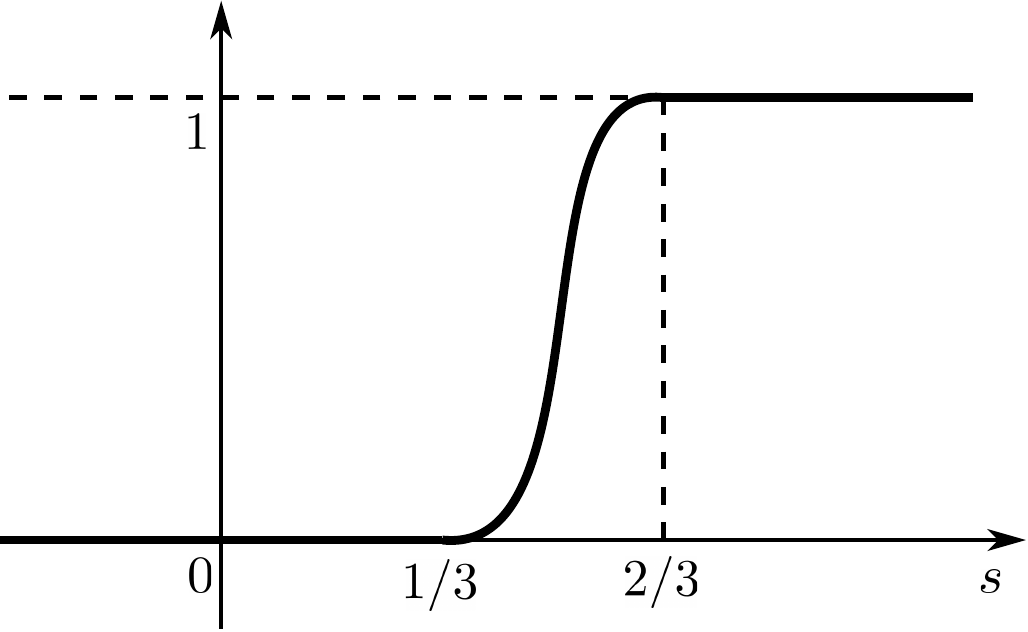}
\caption{The graph of the function $d$.}\label{fg:cutoff2}
\end{figure}

First, we check that $V_{q,s}$ is contained in $\bigcup_{\rho \in P_{\{0\}}} D_\rho$ for any $q \in S_R^1$ and $s \in [0,1]$.
Then, we set $q=R\exp(\sqrt{-1} \theta)$ and consider the projection $p \colon X_{{\mathcal F}'}({\mathbb C}) \times (-\epsilon, 2\pi+\epsilon) \times (0,1) \to (-\epsilon, 2\pi+\epsilon) \times (0,1)$ given by
\begin{gather*}
(x,\theta,s) \mapsto (\theta, s),
\end{gather*}
where $\epsilon \in {\mathbb R}$ is a small constant such that $0<\epsilon \ll 1$.
Let $Y$ be the subset of $X_{{\mathcal F}'}({\mathbb C}) \times (-\epsilon, 2\pi+\epsilon) \times (0,1)$ def\/ined by
\begin{gather*}
Y := \big\{ (x, \theta, s) \in X_{{\mathcal F}'}({\mathbb C}) \times (-\epsilon, 2\pi+\epsilon) \times (0,1) \,|\, \tilde{f}_{q,s}(x)=0 \big\}.
\end{gather*}
We use the following theorem.

\begin{Theorem}[Ehresmann's f\/ibration theorem]\label{th:ehr}
Let $f \colon E \to M$ be a $C^\infty$ map between smooth manifolds. If the map $f$ is a proper submersion, then the map $f$ is a locally trivial fibration.
\end{Theorem}

We check that the function $\tilde{f}_{q,s}$ has $0$ as a regular value on each $D_\mu \cap O_\sigma({\mathbb C})$ for any $q \in S_R^1$ and $s \in [0, 1]$.
Then, it turns out that the restriction of $p$ to $Y$ is a submersion. In addition, we can easily see that $p|_Y$ is proper.
From Theorem~\ref{th:ehr}, we can conclude that the family of submanifolds~$\{V_{q,s}\}_{s \in [0,1]}$ gives an isotopy between $V_q$ and $W_q$.
The condition~3 can be shown by a simple calculation.

\begin{proof}[Proof of Theorem~\ref{th:sub}] We set $T:=\bigcup_{\rho \in P_{\{0\}}} D_{\rho}$. First, we show that $V_{q,s}$ is contained in $T$ for any $q \in S_R^1$ and $s \in [0,1]$. Since we have
\begin{gather*}
O_{\{0\}}({\mathbb C}) = \Bigg( \bigcup_{\substack{\{m_0, \dots m_p\} \subset A \\ p \in {\mathbb Z}^{\geq 0}}} D_{m_0, \dots, m_p} \Bigg) \cap O_{\{0\}}({\mathbb C}),
\end{gather*}
it follows from Lemma~\ref{lm:oth} that it is enough to check that the function $\tilde{f}_{q,s}$ can not be $0$ on $D_m \cap O_{\{0\}}({\mathbb C})$ for any $m \in A$.
The dominant term of $\tilde{f}_{q,s}$ on $D_m \cap O_{\{0\}}({\mathbb C})$ is only ${c_{m}}^{(1-d(s))}q^{v_m}x^m$ and we have
\begin{gather*}
b_{m',s}(x) =
\begin{cases}
1-d(s), & m' \neq m, \\
 1 & m' =m.
\end{cases}
\end{gather*}
Hence, the function $\tilde{f}_{q,s}$ can be written on $D_m \cap O_{\{0\}}({\mathbb C})$ as
\begin{gather*}
{c_{m}}^{(1-d(s))}q^{v_m}x^m+(1-d(s)) \bigg\{\sum_{p} h_pq^{i_p}x^{j_p} \bigg\},
\end{gather*}
where $h_p \in {\mathbb C}$, $i_p \in {\mathbb Z}$, $j_p \in A$ and each term $h_pq^{i_p}x^{j_p}$ denotes other monomial which is not dominant on~$D_m$, i.e., $|q^{i_p}x^{j_p}|/|q^{v_m}x^m| \leq R^{-C_0}$. (Each index $p$ satisf\/ies that either $j_p \neq m$ or $j_p = m$ and $i_p < v_m$.) Hence, for suf\/f\/iciently large~$R$, the function $\tilde{f}_{q,s}$ can not be $0$ on $D_m \cap O_{\{0\}}({\mathbb C})$. Then we have $V_{q,s} \subset T$ for all $q \in S_R^1$ and $s \in [0,1]$. In particular, the condition~2 holds.

Next, we show that the projection $\left. p \right|_Y$ is a proper submersion. For $\mu \in P_\sigma$, let $\mu' \in P_{\{0\}}$ be the cell such that $\mu = \mu' \cap X_{{\mathcal F}',\sigma}$. We def\/ine $m_0, \dots, m_{n+1-k}$ by $\{m_0, \dots, m_{n+1-k}\} =A_{\mu'}$ (the set $A_{\mu'}$ is def\/ined in~\eqref{eq:A}). For any $m \in A \setminus A_{\mu'}$, there exists $m_i \in A_{\mu'}$ such that $|q^{v_{m}}x^{m}|/|{q^{v_{m_i}}x^{m_i}}| \leq R^{-C_0}$ on $D_{\mu'} \cap O_{\{0\}}({\mathbb C})$. Then we have $b\big(\log_R|q^{v_{m}}x^{m}|-\log_R|{q^{v_{m_i}}x^{m_i}}|\big) \equiv 1$ and $b\big(\log_R|q^{v_{m_i}}x^{m_i}|-\log_R|{q^{v_{m}}x^{m}}|\big) \equiv 0$. Therefore, we have
\begin{gather*}
 b_{m,s} |_{D_{\mu'} \cap O_{\{0\}}({\mathbb C})}(x)\\
 \qquad{} =
\begin{cases}
(1-d(s))+d(s)\prod\limits_{i=0}^{n+1-k} b \big( \log_R|q^{v_{m_i}}x^{m_i}|-\log_R|{q^{v_{m}}x^{m}}| \big), & m \in A_{\mu'}, \\
1-d(s) &\mathrm{otherwise,}
\end{cases}
\end{gather*}
and
\begin{gather}
 \tilde{f}_{q,s} \big|_{D_{\mu'} \cap O_{\{0\}}({\mathbb C})}(x)= \sum_{i=0}^{n+1-k} b_{m_i,s}(x)c_{m_i}^{(1-d(s))}q^{v_{m_i}}x^{m_i}\nonumber\\
\hphantom{\tilde{f}_{q,s} \big|_{D_{\mu'} \cap O_{\{0\}}({\mathbb C})}(x)=}{}
+(1-d(s))\bigg\{ f_{q}-\sum_{m \in A}{c_{m}}q^{v_{m}}x^{m}+\sum_{m \in A \setminus A_{\mu'}}{c_{m}}^{(1-d(s))}q^{v_{m}}x^{m}\bigg\}.\label{eq:a}
\end{gather}
Let $\tau \in {\mathcal F}'$ be an $(n+1)$-dimensional cone having $\sigma$ as its face. Let further $e_1, \dots, e_{n+1} \in {\mathbb Z}^{n+1}$ be the primitive generators of $\tau^\vee$. We rearrange $e_1, \dots, e_{n+1}$ if necessary, and set $y_i:=x^{e_i}$ $(i=1,\dots, n+1)$ so that the set of functions $(y_1, \dots, y_{n+1})$ forms a coordinate system on $U_\tau({\mathbb C}) \cong {\mathbb C}^{n+1}$ such that $y_{n+2-l}= \cdots =y_{n+1}=0$ on $O_\sigma({\mathbb C})$. Let $(\tilde{x}_1, \dots, \tilde{x}_{n+1-l})$ be a standard coordinate with respect to $\mu$ such that $\tilde{x}_i:=q^{v_{m_i}}x^{m_i}/q^{v_{m_0}}x^{m_0}$ for $i=1, \dots, n+1-k$. Then, the set of functions $(\tilde{x}_1, \dots, \tilde{x}_{n+1-l}, y_{n+2-l}, \dots, y_{n+1})$ forms a coordinate system on $\bigcup_{\sigma' \prec \sigma} O_{\sigma'}({\mathbb C})$. We def\/ine the functions $b_{m_i,\mu} \colon O_{\{0\}}({\mathbb C}) \to {\mathbb R}$ by
\begin{gather*}
b_{m_0,\mu}(\tilde{x}) :=\prod_{j=1}^{n+1-k}b(\log_R|\tilde{x}_j|),\\
b_{m_i,\mu}(\tilde{x}) :=b(-\log_R|\tilde{x}_i|) \prod_{j=1}^{n+1-k}b(\log_R|\tilde{x}_j|-\log_R|\tilde{x}_i|)\qquad \mathrm{for} \ \ i=1, \dots, n+1-k,
\end{gather*}
and set $b_{m_i,\mu,s}(\tilde{x}) := (1-d(s))+d(s) b_{m_i,\mu}(\tilde{x})$.
In the coordinate system $(\tilde{x}_1, \dots, \tilde{x}_{n+1-l}, y_{n+2-l}$, $\dots, y_{n+1})$, we divide \eqref{eq:a} by $q^{m_0}x^{m_0}$ to obtain
\begin{gather*}
\frac{\tilde{f}_{q,s} \big|_{D_{\mu'} \cap O_{\{0\}}({\mathbb C})} }{q^{m_0}x^{m_0}}=
b_{m_0,\mu,s}(\tilde{x})c_{m_0}^{(1-d(s))}\!+\!\sum_{i=1}^{n+1-k} \!\!b_{m_i,\mu,s}(\tilde{x})c_{m_i}^{(1-d(s))}\tilde{x}_i
+(1-d(s))\{ \mathrm{other\ terms} \}.
\end{gather*}
Notice that other terms are not dominant on $D_{\mu'}$. Hence, we may assume that the subset $V_{q,s}$ is def\/ined on $D_\mu \cap O_\sigma({\mathbb C})$ by
\begin{gather}\label{eq:b}
b_{m_0,\mu,s}(\tilde{x})c_{m_0}^{(1-d(s))}+\sum_{i=1}^{n+1-k} b_{m_i,\mu,s}(\tilde{x})c_{m_i}^{(1-d(s))}\tilde{x}_i
+(1-d(s))\bigg\{ \sum_{p} h_p q^{i_p}\tilde{x}^{j_p} \bigg\}=0,
\end{gather}
where $h_p \in {\mathbb C}$ and terms $\sum_{p} h_p q^{i_p}\tilde{x}^{j_p}$ denote other terms which are not dominant on $D_\mu$. We have $\big|q^{i_p}\tilde{x}^{j_p}\big| \leq R^{-C_0}$ and $\big|q^{i_p}\tilde{x}^{j_p}\big|/|\tilde{x}_{i}| \leq R^{-C_0}$ for all $p$. Let $G(\tilde{x}_1, \dots, \tilde{x}_{n+1-l})$ denote the left-hand side of~\eqref{eq:b}. We show that $G(\tilde{x}_1, \dots, \tilde{x}_{n+1-l})$ has $0 \in {\mathbb C}$ as a regular value on $D_\mu \cap O_\sigma({\mathbb C})$.

We def\/ine the subset $D_{\mu,m_i}:=\{ x \in D_{\mu} \,|\, |q^{v_{m_i}}x^{m_i}|/|q^{v_{m_j}}x^{m_j}| \geq1\ (j=0, \dots, n+1-k) \}$ for $i=0, \dots, n+1-k$. For any $x \in D_\mu$, there exists $i \in \{1, \dots n+1-k\}$ such that $|q^{v_{m_i}}x^{m_i}| \geq |q^{v_{m_j}}x^{m_j}|$ for any $j \in \{1, \dots, n+1-k\}$.
Then we have $D_{\mu} = \bigcup_{i=0}^{n+1-k} D_{\mu,m_i}$. We have only to show that the Jacobian matrix of $G(\tilde{x}_1, \dots, \tilde{x}_{n+1-l})$ has the maximal rank on $D_{\mu,{m_1}}$. On $D_{\mu,{m_1}}$, we have $b_{m_1,\mu,s}(\tilde{x}) \equiv 1$. We set $\tilde{x}_i=r_i\exp(\sqrt{-1} \theta_i)\ (r_i \in {\mathbb R}^{\geq0}, \theta_i \in [0,2\pi])$ and let $M$ be a~$2 \times 2$ matrix def\/ined by
\begin{gather*}
M:=\begin{pmatrix}
\dfrac{\partial}{\partial r_{1}} \operatorname{Re}(G) &\dfrac{\partial}{\partial \theta_{1}} \operatorname{Re}(G)\vspace{1mm}\\
\dfrac{\partial}{\partial r_{1}} \operatorname{Im}(G) &\dfrac{\partial}{\partial \theta_{1}} \operatorname{Im}(G)
\end{pmatrix}.
\end{gather*}
We can show $\det(M) \neq 0$ for a suf\/f\/iciently large $R$ by the concrete calculation. Hence, the subsets $V_{q,s}$ and $Y$ is smooth submanifold in $X_{{\mathcal F}'}({\mathbb C})$ and $\{X_{{\mathcal F}'}({\mathbb C}) \times (-\epsilon, 2\pi+\epsilon) \times (0,1)\}$ respectively. Moreover, it turns out that the projection $p|_Y \colon Y \to (-\epsilon, 2\pi+\epsilon) \times (0,1)$ is a~sub\-mersion.

In addition, for any compact subset $C \subset (-\epsilon, 2\pi+\epsilon) \times (0,1)$, the inverse image $(p|_Y)^{-1}(C) \subset Y$ coincides with $\{(x, \theta, s) \in X_{{\mathcal F}'}({\mathbb C}) \times C \,|\, \tilde{f}_{q,s}(x)=0\}$. Then the set $(p|_Y)^{-1}(C)$ is compact and the map $p|_Y$ is proper.
Hence, it turns out from Theorem~\ref{th:ehr} that the map $p|_Y$ has a structure of a f\/iber bundle with the f\/iber $V_{R,1}=W_{q=R}=:W_R$. Therefore, the family of submanifolds $\{V_{q,s}\}_{s \in [0,1]}$ gives an isotopy and the condition~1 holds.

Finally, we check the condition~3. In~\eqref{eq:b}, we set $s=1$ to obtain
\begin{gather*}
\prod_{j=1}^{n+1-k}b(\log_R|\tilde{x}_j|)+\sum_{i=1}^{n+1-k}\bigg\{b(-\log_R|\tilde{x}_i|) \prod_{j=1}^{n+1-k}b(\log_R|\tilde{x}_j|-\log_R|\tilde{x}_i|) \bigg\}\tilde{x}_i=0.
\end{gather*}
This coincides with the def\/ining function of the $(n-k)$-dimensional tropically localized hyperplane in $(\tilde{x}_1, \dots, \tilde{x}_{n+1-k})$ and the left-hand side is independent of the values of $\tilde{x}_{n+2-k}, \dots$, $\tilde{x}_{n+1-l}$.
Hence, the condition~3 holds.
\end{proof}

\section{Monodromy transformations}\label{sc:4}

We use the same notation as in Section~\ref{sc:3} and keep the assumption that $V(\operatorname{trop}(F))$ is smooth. We set $W_R:=W_{q=R}$. Let $\{\psi_{q=R\exp(\sqrt{-1}\theta)} \colon W_R \to W_q\}_{\theta \in [0,2\pi]}$ be a family of homeomorphisms which depends on $\theta$ continuously. It is clear that the map $\psi_{q=R\exp(2 \pi \sqrt{-1})} \colon W_R \to W_R$ gives the monodromy transformation of $\{V_q\}_{q \in S_R^1}$ under the identif\/ication $W_R \cong V_R$. Hence, it is suf\/f\/icient to construct a monodromy transformation of $\{W_q\}_{q \in S_R^1}$ in order to get that of $\{V_q\}_{q \in S_R^1}$.

\begin{Proposition}\label{pr:phi} There exists a continuous map $\phi \colon \operatorname{Log}_R(W_R) \to V(\operatorname{trop}(F))$ satisfying the following condition:
\begin{itemize}\itemsep=0pt
 \item[$(\ast)$] $\phi(\operatorname{Log}_R(W_R) \cap \widehat{D}_{\rho} \cap O_\sigma({\mathbb T})) \subset \rho \cap O_\sigma({\mathbb T})$ for any $\sigma \in {\mathcal F}'$ and $\rho \in P_\sigma$.
\end{itemize}
Moreover, such maps are unique up to homotopy.
\end{Proposition}
\begin{proof}
For each cell $\rho \in P$, we construct a continuous map $\phi_\rho \colon \operatorname{Log}_R(W_R) \cap \widehat{D}_{\rho} \to V(\operatorname{trop}(F))$ satisfying following conditions:
\begin{enumerate}\itemsep=0pt
\item[(i)] $\phi_{\rho}(\operatorname{Log}_R(W_R) \cap \widehat{D}_\rho \cap O_\sigma({\mathbb T})) \subset \rho \cap O_\sigma({\mathbb T})$, where $\sigma \in {\mathcal F}'$ is a cone such that $\rho \in P_\sigma$. 
\item[(ii)] For any face $\mu \prec \rho$, the map $\phi_\rho$ coincides with $\phi_\mu$ on $\operatorname{Log}_R(W_R) \cap \widehat{D}_\rho \cap \widehat{D}_\mu$. 
\end{enumerate}
We construct $\phi_\rho$ in an ascending order of $\dim \rho$ as follows. For each vertex $\rho \in P$, we set $\phi_\rho$ as a~constant map from $\operatorname{Log}_R(W_R) \cap \widehat{D}_\rho$ to $\rho$. For each $1$-cell $\rho$, let $\nu_0$ and $\nu_1$ be the endpoints of~$\rho$. We set each $\phi_\rho$ as a continuous map to $\rho$ so that $\phi_\rho$ coincides with the constant map to $\nu_i$ on $\operatorname{Log}_R(W_R) \cap \widehat{D}_\rho \cap \widehat{D}_{\nu_i}$ and satisf\/ies the condition~(i). Assume that we have constructed $\phi_\rho$ for all cells whose dimensions are lower than $k-1$. For each $k$-cell $\rho$, we def\/ine $\phi_\rho$ as a continuous map to $\rho$ so that $\phi_\rho$ coincides with $\phi_\mu$ on $\operatorname{Log}_R(W_R) \cap \widehat{D}_\rho \cap \widehat{D}_{\mu}$ for any face $\mu$ of $\rho$ and satisf\/ies the condition~(i). In this way, we can construct a family of maps $\{\phi_\rho\}_{\rho \in P}$ such that each map~$\phi_\rho$ satisf\/ies the condition~(i) and~(ii).

\begin{Example} Consider the polynomial $F=1+x_1+x_2$. Fig.~\ref{fg:phi} shows $V(\operatorname{trop}(F))$ and $\operatorname{Log}_R(W_R)$. Let~$\nu$ denote the center vertex of $V(\operatorname{trop}(F))$. The region colored gray denotes $\operatorname{Log}_R(W_R) \cap \widehat{D}_\nu$. The map $\phi_\nu \colon \operatorname{Log}_R(W_R) \cap \widehat{D}_\nu \to V(\operatorname{trop}(F))$ is the constant map to $\nu$ as shown in Fig.~\ref{fg:phi2}.
\begin{figure}[t]\centering
\begin{minipage}[t]{0.47\hsize}\centering
\includegraphics[width=40mm]{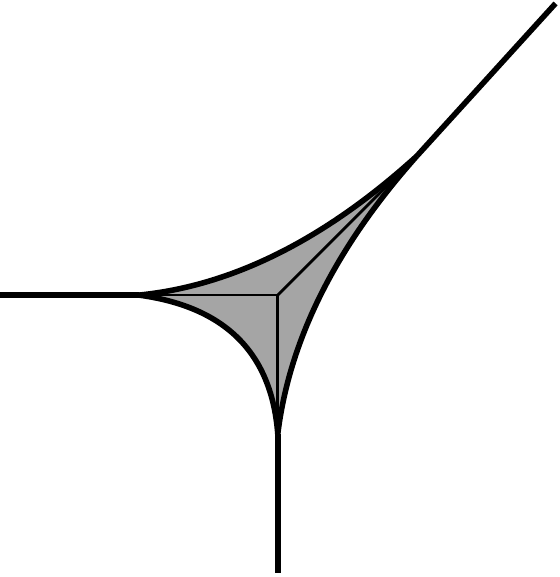}
\caption{The region $\operatorname{Log}_R(W_R)$ for $F=1+x_1+x_2$.}\label{fg:phi}
\end{minipage}\quad
\begin{minipage}[t]{0.47\hsize}\centering
\includegraphics[width=40mm]{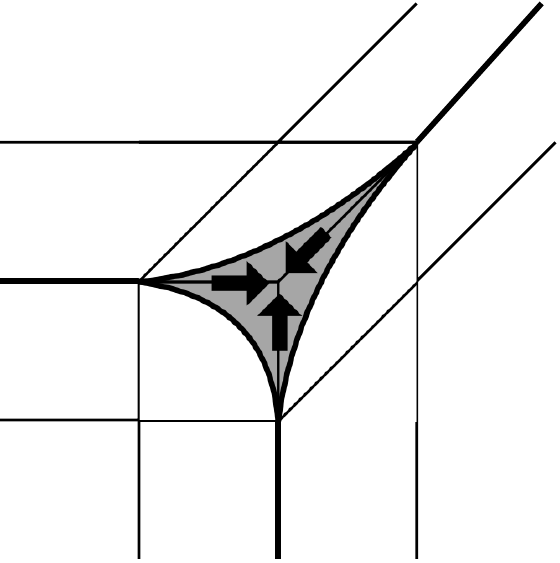}
\caption{The map $\phi_\nu$ for $F=1+x_1+x_2$.}\label{fg:phi2}
\end{minipage}
\end{figure}
\end{Example}

For any $\sigma \in {\mathcal F}'$ and $\rho_1, \rho_2 \in P_\sigma$, if $\widehat{D}_{\rho_1} \cap \widehat{D}_{\rho_2} \neq \varnothing$, there exists $\rho \in P_\sigma$ such that $\rho \prec \rho_1, \rho_2$ and $\widehat{D}_{\rho_1} \cap \widehat{D}_{\rho_2} \subset \widehat{D}_\rho$ (Lemma~\ref{lm:usig}).
From the condition (ii), the map $\phi_{\rho_i}$ coincides with $\phi_\rho$ on $\operatorname{Log}_R(W_R) \cap \widehat{D}_\rho \cap \widehat{D}_{\rho_i}$ $(i=1, 2)$. Hence, the maps $\phi_{\rho_1}$ and $\phi_{\rho_2}$ coincide with each other on $\operatorname{Log}_R(W_R) \cap \widehat{D}_\rho \cap \widehat{D}_{\rho_1} \cap \widehat{D}_{\rho_2}$=$\operatorname{Log}_R(W_R) \cap \widehat{D}_{\rho_1} \cap \widehat{D}_{\rho_2}$. Then it turns out that we can get the continuous map $\phi \colon \operatorname{Log}_R(W_R) \to V(\operatorname{trop}(F))$ by gluing $\{\phi_\rho\}_\rho$. The map $\phi$ satisf\/ies the condition $(\ast)$.

Let $\phi_0, \phi_1 \colon \operatorname{Log}_R(W_R) \to V(\operatorname{trop}(F))$ be two maps satisfying the condition $(\ast)$. We construct a~family of continuous maps $\{ \phi_s \colon \operatorname{Log}_R(W_R) \to X_{{\mathcal F}'}({\mathbb T}) \}_{s \in [0,1]}$ by $\phi_s(X):=(1-s)\phi_0(X)+s\phi_1(X)$, where the addition and the multiplications are taken on~$O_\sigma({\mathbb T}) \cong {\mathbb R}^{l}$ for the cone $\sigma \in {\mathcal F}'$ such that $X \in O_\sigma({\mathbb T})$. This construction is independent of the choice of coordinates on~$O_\sigma({\mathbb T})$. Since each cell is convex, $\phi_s$ satisf\/ies the condition~$(\ast)$ for any~$s \in [0,1]$. Therefore, each map $\phi_s$ is well-def\/ined as a continuous map to $V(\operatorname{trop}(F))$ and $\{\phi_s\}_{s \in [0,1]}$ gives a homotopy between $\phi_0$ and $\phi_1$.
\end{proof}

We f\/ix a map $\phi \colon \operatorname{Log}_R(W_R) \to V(\operatorname{trop}(F))$ satisfying the condition $(\ast)$ in Proposition~\ref{pr:phi}. Let $\sigma \in {\mathcal F}'$ be an $l$-dimensional cone. We choose $a_i, b_{ij} \in {\mathbb Z}$ $(i=1, \dots, n+1-l$, $j=1, \dots, n+1)$ so that the sets of functions $(y_1, \dots ,y_{n+1-l})$ and $(Y_1, \dots, Y_{n+1-l})$ def\/ined by
\begin{gather}\label{eq:cood}
y_i:=q^{a_i}\prod_{j=1}^{n+1}x_j^{b_{ij}},\qquad Y_i:=a_i+\sum_{j=1}^{n+1}b_{ij}X_j,
\end{gather}
form coordinate systems on $O_\sigma({\mathbb C})$ and $O_\sigma({\mathbb T})$, respectively. For each $q = R \exp(\sqrt{-1}\theta) \in S_R^1$, we def\/ine the map $\psi_{\sigma,q} \colon W_R \cap O_\sigma({\mathbb C}) \to O_\sigma({\mathbb C})$ by
\begin{gather*}
(y_1, \dots ,y_{n+1-l}) \mapsto \big(\tilde{\phi}_{1,\theta}y_1, \dots, \tilde{\phi}_{n+1-l,\theta}y_{n+1-l}\big),
\end{gather*}
where $\tilde{\phi}_{i,\theta}(y_1, \dots, y_{n+1-l}) := \exp \left(\sqrt{-1} \theta Y_i \circ \phi \circ \operatorname{Log}_R (y) \right)$ $(i=1, \dots, n+1-l)$.

\begin{Lemma}\label{lm:c0}
For any $q \in S_R^1$ and $\sigma \in {\mathcal F}'$, the map $\psi_{\sigma,q}$ is independent of the choice of the coordinate system $(y_1, \dots, y_{n+1-l})$ on $O_\sigma({\mathbb C})$ and the image $\psi_{\sigma,q}(W_R \cap O_\sigma({\mathbb C}))$ is contained in $W_q \cap O_\sigma({\mathbb C})$.
\end{Lemma}
\begin{proof}
First, we show that the map $\psi_{\sigma,q}$ is independent of the choice of the coordinate. Let $(z_1, \dots, z_{n+1-l})$ and $(Z_1, \dots, Z_{n+1-l})$ be other coordinate systems on $O_\sigma({\mathbb C})$ and $O_\sigma({\mathbb T})$ def\/ined just as $(y_1, \dots, y_{n+1-l})$ and $(Y_1, \dots, Y_{n+1-l})$.
We can write
\begin{gather*}
z_i=q^{\alpha_i}\prod_{j=1}^{n+1-l}y_j^{\beta_{ij}}, \qquad y_i=\prod_{j=1}^{n+1-l}\big(q^{-\alpha_j}z_j\big)^{\gamma_{ij}},
\end{gather*}
where $\alpha_i, \beta_{ij}$ are some integral numbers. Here, we have $\sum\limits_{j=1}^{n+1-l}\beta_{ij}\gamma_{jk}=\delta_{ik}$. Let $\psi_{\sigma,q}' \colon W_R \cap O_\sigma({\mathbb C}) \to O_\sigma({\mathbb C})$ be the map def\/ined in $(z_1, \dots, z_{n+1-l})$. For all $z=(z_1, \dots, z_{n+1-l})=y=(y_1, \dots, y_{n+1-l}) \in W_R \cap O_\sigma({\mathbb C})$, we have
\begin{gather*}
z_i \left( \psi_{\sigma,q}' (z) \right)= \left\{ \exp \left( \sqrt{-1} \theta \left( \alpha_i+\sum_{j=1}^{n+1-l} \beta_{ij} Y_j \right) \left( \phi \circ \operatorname{Log}_R (z)\right) \right) \right\} R^{\alpha_i} \prod_{j=1}^{n+1-l} y_j^{\beta_{ij}},
\end{gather*}
and
\begin{gather*}
y_i \left( \psi_{\sigma,q}'(z) \right) =\prod_{k=1}^{n+1-l} \big( q^{-\alpha_k} z_k (\psi_{\sigma,q}' (z)) \big)^{\gamma_{ik}} \\
 =\prod_{k=1}^{n+1-l} \left[ q^{-\alpha_k} \left\{ \exp \left( \sqrt{-1} \theta \left( \alpha_k+\sum_{j=1}^{n+1-l} \beta_{kj} Y_j \right) \left( \phi \circ \operatorname{Log}_R (z)\right) \right) \right\} R^{\alpha_k} \prod_{j=1}^{n+1-l} y_j^{\beta_{kj}} \right]^{\gamma_{ik}} \\
 =\left\{ \exp \left(\sqrt{-1} \theta \sum_{k,j} \gamma_{ik} \beta_{kj} Y_j \left(\phi \circ \operatorname{Log}_R (y) \right) \right) \right\} \prod_{j=1}^{n+1-l} y_j^{\sum_{k} \gamma_{ik} \beta_{kj}} \\
 =\exp \left( \sqrt{-1} \theta Y_i \left( \phi \circ \operatorname{Log}_R (y) \right) \right) y_i = \tilde{\phi}_{i,\theta} y_i.
\end{gather*}
Therefore, we have $y_i ( \psi_{\sigma,q}'(y) )=y_i (\psi_{\sigma,q}(y) )=\tilde{\phi}_{i,\theta} y_i$. Hence, one has $\psi_{\sigma,q} = \psi_{\sigma,q}'$.

Next, we show that the image $\psi_{\sigma,q}(W_R \cap O_\sigma({\mathbb C}))$ is contained in $W_q \cap O_\sigma({\mathbb C})$. Let $\mu \in P_\sigma$ be a cell such that $\mu = \mu' \cap X_{{\mathcal F}',\sigma}({\mathbb C})$ for a $k$-cell $\mu' \in P_{\{0\}}$ and $(\tilde{x}_1, \dots, \tilde{x}_{n+1-l})$ be a standard coordinate with respect to $\mu$. Since $\widetilde{X}_1=\dots=\widetilde{X}_{n+1-k}=0$ on $\mu$, the restriction of $\psi_{\sigma,q}$ to $D_\mu \cap O_\sigma({\mathbb C})$ coincides with
\begin{gather*}
(\tilde{x}_1, \dots, \tilde{x}_{n+1-l}) \mapsto
\big(\tilde{x}_1, \dots, \tilde{x}_{n+1-k}, \tilde{\phi}_{n+2-k,\theta}\tilde{x}_{n+2-k}, \dots, \tilde{\phi}_{n+1-l,\theta} \tilde{x}_{n+1-l}\big).
\end{gather*}
Since the def\/ining equation of $W_q$ on $D_\mu \cap O_\sigma({\mathbb C})$ coincides with that of the $(n-k)$-dimensional tropically localized hyperplane in $(\tilde{x}_1, \dots, \tilde{x}_{n+1-k})$, we have $\psi_{\sigma,q} \left(W_R \cap O_\sigma({\mathbb C}) \cap D_\mu \right) \subset W_q \cap O_\sigma({\mathbb C}) \cap D_\mu$. Hence, the map $\psi_{\sigma,q}$ is well-def\/ined as a map from $W_R \cap O_\sigma({\mathbb C})$ to $W_q \cap O_\sigma({\mathbb C})$.
\end{proof}

\begin{Lemma}\label{lm:c1}
For any $q \in S_R^1$, the family of maps $\{\psi_{\sigma,q}\}_{\sigma \in {\mathcal F}'}$ glues together to give the homeomorphism $\psi_{q} \colon W_R \to W_q$.
\end{Lemma}
\begin{proof}
Let $\tau \in {\mathcal F}'$ be an $(n+1)$-dimensional cone. Let further $e_1, \dots, e_{n+1} \in {\mathbb Z}^{n+1}$ be the primitive generators of $\tau^\vee$. We set $w_i:=x^{e_i}$ and $W_i:=e_i \cdot X$ $(i=1,\dots, n+1)$. Then the set of functions $(w_1, \dots, w_{n+1})$ and $(W_1, \dots, W_{n+1})$ form coordinate systems on $U_\tau({\mathbb C}) \cong {\mathbb C}^{n+1}$ and $U_\tau({\mathbb T}) \cong {\mathbb T}^{n+1}$, respectively. We def\/ine the map $\Psi_{\tau,q} \colon W_R \cap U_\tau({\mathbb C}) \to U_\tau({\mathbb C})$ by
\begin{gather*}
(w_1, \dots ,w_{n+1}) \mapsto \big(\tilde{\phi}_{1,\theta}w_1, \dots, \tilde{\phi}_{n+1,\theta}w_{n+1}\big),
\end{gather*}
where $\tilde{\phi}_{i,\theta}(w_1, \dots, w_{n+1}) := \exp \left(\sqrt{-1} \theta \, W_i \circ \phi \circ \operatorname{Log}_R (w) \right)\ (i=1, \dots, n+1)$. It is clear from Lemma~\ref{lm:c0} that the map $\psi_{\sigma,q}$ coincides with $\Psi_{\tau,q}$ on $O_\sigma({\mathbb C}) \subset U_\tau({\mathbb C})$ for any face \mbox{$\sigma \prec \tau$}. Therefore, the family of maps $\{\psi_{\sigma,q}\}_{\sigma \in {\mathcal F}'}$ glues together to give the continuous map \mbox{$\psi_{q} \colon W_R \to W_q$}. In addition, the inverse map $(\psi_{\sigma,q})^{-1} \colon W_q \to W_R$ is given by
\begin{gather*}
(y_1, \dots ,y_{n+1-l}) \mapsto \big(\tilde{\phi}_{1,-\theta}y_1, \dots, \tilde{\phi}_{n+1-l,-\theta}y_{n+1-l}\big)
\end{gather*}
and $\{ (\psi_{\sigma,q})^{-1}\}_{\sigma \in {\mathcal F}'}$ forms the inverse map $(\psi_q)^{-1} \colon W_q \to W_R$. It is obvious that the map $(\psi_q)^{-1}$ is continuous. Hence, the map $\psi_q$ is a homeomorphism for any $q \in S_R^1$.
\end{proof}

The following is the main theorem of this paper.

\begin{Theorem}\label{th:main} Assume that $V(\operatorname{trop}(F))$ is smooth. We fix a sufficiently large number $R \in {\mathbb R}^{>0}$. Let $\phi \colon \operatorname{Log}_R(W_R) \to V(\operatorname{trop}(F))$ be a map satisfying the condition $(\ast)$ in Proposition~{\rm \ref{pr:phi}}. Let further $\psi \colon W_R \to W_R$ be the map defined on each orbit $O_\sigma({\mathbb C})$ $(\sigma \in {\mathcal F}')$ by
\begin{gather}\label{eq:c4}
(y_1, \dots ,y_{n+1-l}) \mapsto \big(\tilde{\phi}_{1}y_1, \dots, \tilde{\phi}_{n+1-l}y_{n+1-l}\big),
\end{gather}
where $(y_1, \dots ,y_{n+1-l})$ is a coordinate system on $O_\sigma({\mathbb C})$ defined as in~\eqref{eq:cood} and
\begin{gather*}
\tilde{\phi}_{i}(y) := \exp \big(2 \pi \sqrt{-1} Y_i \circ \phi \circ \operatorname{Log}_R (y) \big).
\end{gather*}
Then the map $\psi \colon W_R \to W_R$ gives a monodromy transformation of $\{ V_q \}_{q \in S_R^1}$ under the identification $V_R \cong W_R$.
For each cell $\mu \in P_\sigma$, the restriction of $\psi$ to $D_\mu \cap O_\sigma({\mathbb C})$ coincides with
\begin{gather*}
(\tilde{x}_1, \dots, \tilde{x}_{n+1-l}) \mapsto
\big(\tilde{x}_1, \dots, \tilde{x}_{n+1-k}, \tilde{\phi}_{n+2-k}\tilde{x}_{n+2-k}, \dots, \tilde{\phi}_{n+1-l} \tilde{x}_{n+1-l}\big),
\end{gather*}
in a standard coordinate $(\tilde{x}_1, \dots, \tilde{x}_{n+1-l})$ with respect to~$\mu$.
\end{Theorem}
\begin{proof}
It is clear from Lemmas~\ref{lm:c0} and~\ref{lm:c1}.
\end{proof}

\section{Proof of Corollary~\ref{cr:iwa}}\label{sc:5}

In this section, we show that Corollary~\ref{cr:iwa} follows from Theorem~\ref{th:main}. We set $n=1$. Let $\phi \colon \operatorname{Log}_R(W_R) \to V(\operatorname{trop}(F))$ be a map satisfying the condition $(\ast)$ in Proposition~\ref{pr:phi}. We set the map $\phi$ so that the restriction of $\phi$ to $\operatorname{Log}_R(W_R) \cap \widehat{D}_\rho$ gives a bijection to $\rho$ for any edge $\rho \in P_{\{0\}}$. Let $\nu \in P_{\{0\}}$ be a vertex of $V(\operatorname{trop}(F))$ contained in $O_{\{0\}}({\mathbb T})$. Let further $(\tilde{x}_1,\tilde{x}_2)$ and $(\widetilde{X}_1, \widetilde{X}_2)$ be standard coordinates with respect to $\nu$ (see Section~\ref{sc:2.3}). On $D_\nu$, the tropical localization $W_q$ is def\/ined by the def\/ining equation of the $1$-dimensional tropically localized hyperplane in $(\tilde{x}_1, \tilde{x}_2)$. Since we have $\widetilde{X}_1(\nu)=\widetilde{X}_2(\nu)=0$ and the restriction of $\phi$ to $\operatorname{Log}_R(W_R) \cap \widehat{D}_\nu$ is the constant map to $\nu$, the monodromy transformation $\psi$ in Theorem~\ref{th:main} coincides with the identity map on $D_\nu$. Similarly, it turns out that the map $\psi$ also coincides with the identity map on $D_{\nu'}$ for any vertex $\nu' \in P$ contained in a lower dimensional torus orbit.

Let $\mu \in P_{\{0\}}$ be a bounded edge of $V(\operatorname{trop}(F))$ and $\nu_1$, $\nu_2$ be the endpoints of $\mu$. We set $\{m_0, m_1, m_2\} \subset A$ so that $\{m_0, m_1, m_2\}=A_{\nu_1}$ and $\{m_0, m_1\}=A_\mu$, where~$A_{\nu_1}$ and $A_{\mu}$ are subsets of $A$ def\/ined in~\eqref{eq:A}. We def\/ine the standard coordinate with respect to $\nu_1$ by
\begin{gather*}
\tilde{x}_i :=q^{v_{m_i}}x^{m_i} / q^{v_{m_0}}x^{m_0}, \qquad \widetilde{X}_i:=(v_{m_i}+m_iX)-(v_{m_0}+m_0X),
\end{gather*}
for $i=1,2$. Then the coordinate systems $(\tilde{x}_1,\tilde{x}_2)$ and $(\widetilde{X}_1, \widetilde{X}_2)$ are also standard coordinates with respect to~$\mu$. On~$D_\mu$, the def\/ining equation of the tropical localization~$W_q$ coincides with
\begin{gather}\label{eq:tr1}
b(\log_R|\tilde{x}_1|)+b(-\log_R|\tilde{x}_1|)\tilde{x}_1=0.
\end{gather}

\begin{Lemma} The solution of \eqref{eq:tr1} is $\tilde{x}_1=-1$.
\end{Lemma}
\begin{proof}The equation \eqref{eq:tr1} coincides with $b(\log_R|\tilde{x}_1|)+\tilde{x}_1=0$ when $C_1 \leq \log_R|\tilde{x}_1| \leq C_0$ and $1+b(-\log_R|\tilde{x}_1|)\tilde{x}_1=0$ when $-C_0 \leq \log_R|\tilde{x}_1| \leq -C_1$. These equations have no solution when~$R$ is suf\/f\/iciently large.
In the case $-C_1 \leq \log_R|\tilde{x}_1| \leq C_1$, \eqref{eq:tr1} coincides with $1+\tilde{x}_1=0$.
\end{proof}

Hence, the tropical localization $W_q$ coincides with the cylinder def\/ined by $\tilde{x}_1=-1$ and $\tilde{x}_2$ are free on~$D_\mu$. Let $l \in {\mathbb Z}_{>0}$ be the length of~$\mu$. In the coordinate system $(\tilde{x}_1,\tilde{x}_2)$, we have $\widetilde{X}_1(\nu_1)=\widetilde{X}_2(\nu_1)=0$ and
$\widetilde{X}_1(\nu_2)=0$, $\widetilde{X}_2(\nu_2)=-l$. Note that the lengths of edges are invariant under the coordinate transformations. Since the restriction of $\phi$ to $\operatorname{Log}_R(W_R) \cap \widehat{D}_\mu$ gives a~bijection to~$\mu$, we can see from Theorem~\ref{th:main} that the map $\psi$ coincides with the composition of $l$-times of Dehn twists on~$D_\mu$. Similarly, it turns out that the restriction of $\psi$ to $D_{\mu'}$ coincides with the compositions of inf\/initely many times of Dehn twists for any unbounded edge $\mu' \in P_{\{0\}}$.

\section{Examples}\label{sc:6}

\subsection{Example in dimension 1}

Consider the polynomial $F$ given by \eqref{eq:hypellip}. The tropical hypersurface $V(\operatorname{trop}(F))$ is shown in Fig.~\ref{fg:ex1dim}. Let $\nu_1$, $\nu_2$, $\mu$ denote cells of $V(\operatorname{trop}(F))$ as shown in Fig.~\ref{fg:ex1dim} ($\nu_1$, $\nu_2$ denote $0$-cells and~$\mu$ denotes the $1$-cell). The regions $\widehat{D}_{\nu_1}$, $\widehat{D}_{\nu_2}$ and $\widehat{D}_{\mu}$ def\/ined in Def\/inition~\ref{df:nbd} are shown in Fig.~\ref{fg:ex1dim2}.

\begin{figure}[t]
\begin{minipage}[t]{0.5\hsize}\centering
\includegraphics[width=25mm]{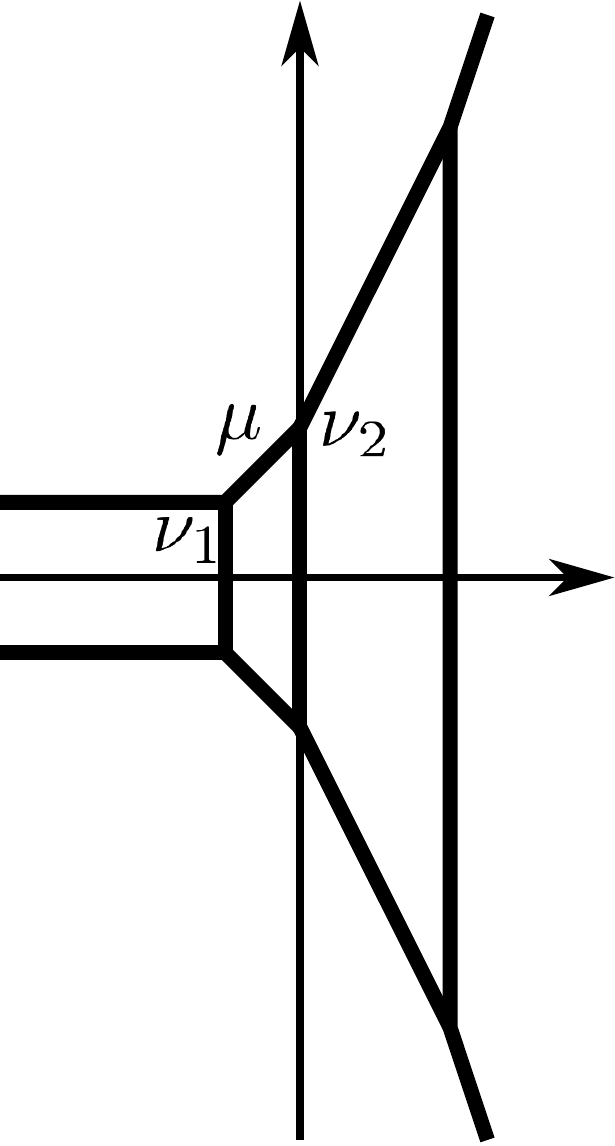}
\caption{The tropical hypersurface $V(\operatorname{trop}(F))$.}\label{fg:ex1dim}
\end{minipage}\quad
\begin{minipage}[t]{0.46\hsize}\centering
\includegraphics[width=25mm]{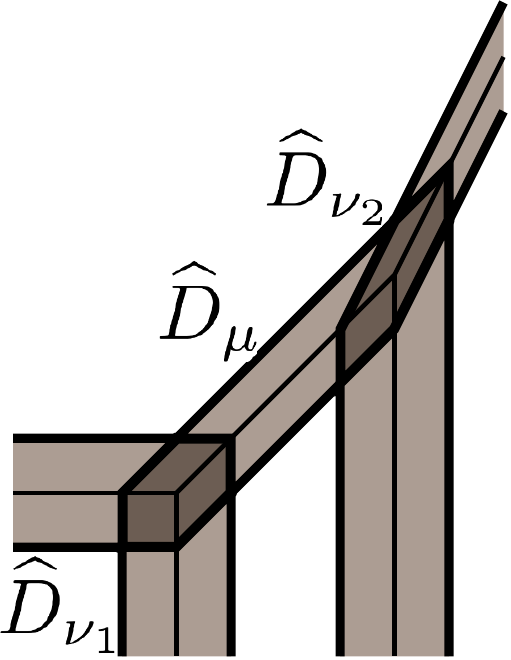}
\caption{Regions $\widehat{D}_{\nu_1}$, $\widehat{D}_{\nu_2}$ and $\widehat{D}_{\mu}$.}\label{fg:ex1dim2}
\end{minipage}
\end{figure}

The set $A_{\nu_1}$ def\/ined in \eqref{eq:A} is given by $\{ (0,2), (1,1), (0,1) \}$. We set
\begin{gather*}
\tilde{x}_1 := q^2x_1x_2/ x_2^2=q^2x_1x_2^{-1}, \qquad \tilde{x}_2:= qx_2/x_2^2=qx_2^{-1}, \\
\widetilde{X}_1 := (2+X_1+X_2)-(2X_2)=2+X_1-X_2, \qquad \widetilde{X}_2:=(1+X_2)-2X_2=1-X_2.
\end{gather*}
Then the sets of functions $(\tilde{x}_1, \tilde{x}_2)$ and $(\widetilde{X}_1, \widetilde{X}_2)$ form standard coordinates on $O_{\{0\}}({\mathbb C})$ and $O_{\{0\}}({\mathbb T})$ with respect to $\nu_1$ def\/ined in Section~\ref{sc:2.3}. Similarly, we have $A_{\nu_2}=\{(0,2), (1,1), (2,1) \}$ and we set
\begin{gather*}
\tilde{x}_3:= q^2x_1^2x_2/x_2^2=q^2x_1^2x_2^{-1}, \qquad \widetilde{X}_3:=(2+2X_1+X_2)-(2X_2)=2+2X_1-X_2,
\end{gather*}
so that the sets of functions $(\tilde{x}_1, \tilde{x}_3)$ and $(\widetilde{X}_1, \widetilde{X}_3)$ form standard coordinates on $O_{\{0\}}({\mathbb C})$ and $O_{\{0\}}({\mathbb T})$ with respect to $\nu_2$. In addition, we have $A_{\mu}=\{ (0,2), (1,1) \}$. Hence, the sets of functions $(\tilde{x}_1, \tilde{x}_2)$, $(\widetilde{X}_1, \widetilde{X}_2)$ and $(\tilde{x}_1, \tilde{x}_3)$, $(\widetilde{X}_1, \widetilde{X}_3)$ also form standard coordinates with respect to~$\mu$. Let~$W_R$ be the tropical localization def\/ined in Def\/inition~\ref{df:loc} and $\phi \colon \operatorname{Log}_R(W_R) \to V(\operatorname{trop}(F))$ be a map satisfying the condition~$(\ast)$ in Proposition~\ref{pr:phi}. Here, we set the map~$\phi$ so that the restriction of $\phi$ to $\operatorname{Log}_R(W_R) \cap \widehat{D}_\rho$ gives a bijection to~$\rho$ for any edge~$\rho \in P_{\{0\}}$. The manifold $W_R$, the map $\phi$ and the monodromy transformation $\psi \colon W_R \to W_R$ on each region are listed in Table~\ref{tb:1}.

\begin{table}[t]\centering
\caption{Monodromy transformation $\psi$ on each region.}\label{tb:1}\vspace{1mm}
\begin{tabular}{@{}cccc@{}}
\hline
region & $D_{\nu_1}$ & $D_{\nu_2}$ & $D_{\mu}$ \\
\hline
standard coordinate & $(\tilde{x}_1, \tilde{x}_2)$ & $(\tilde{x}_1, \tilde{x}_3)$ & $(\tilde{x}_1, \tilde{x}_2)$ or $(\tilde{x}_1, \tilde{x}_3)$ \tsep{2pt}\bsep{3pt}\\
\shortstack{tropical localization \\ $W_R$} & \shortstack{$1$-dimensional tropically \\ localized hyperplane \\ in $(\tilde{x}_1, \tilde{x}_2)$}
& \shortstack{$1$-dimensional tropically \\ localized hyperplane \\ in $(\tilde{x}_1, \tilde{x}_3)$}
& \shortstack{$1$-dimensional \\ cylinder \\ in $\tilde{x}_2$ or $\tilde{x}_3$}\bsep{3pt}\\
map $\phi$ & constant map to $\nu_1$ & constant map to $\nu_2$ & bijection to $\mu$ \bsep{3pt} \\
monodromy $\psi$ & identity map & identity map & Dehn twist\bsep{2pt} \\
\hline
\end{tabular}
\end{table}

Since $\widetilde{X}_1(\nu_1)=\widetilde{X}_2(\nu_1)=0$ and $\widetilde{X}_1(\nu_2)=\widetilde{X}_3(\nu_2)=0$, we can see from Theorem~\ref{th:main} that the restrictions of $\psi$ to $D_{\nu_1}$ and $D_{\nu_2}$ are identity maps. On $D_\mu$, if we use $(\tilde{x}_1, \tilde{x}_2)$ as a coordinate system, we have $\widetilde{X}_1 \equiv 0$ on $\mu$ and $\widetilde{X}_2(\nu_1)=0, \widetilde{X}_2(\nu_2)=-1$. Then we can also see from Theorem~\ref{th:main} that the restriction of~$\psi$ to~$D_{\mu}$ coincides with the Dehn twist in the component of the cylinder in~$\tilde{x}_2$.

\subsection{Example in dimension 2}

Consider the polynomial $G(x_1,x_2,x_3)=t^{-1}+x_1+x_2+x_3+x_1^{-1}x_2^{-1}x_3^{-1}$. Then we have
\begin{gather*}
g_q(x_1,x_2,x_3) =q+x_1+x_2+x_3+x_1^{-1}x_2^{-1}x_3^{-1},\\
\operatorname{trop}(G)(X_1,X_2,X_3) =\max \{ 1, X_1, X_2, X_3, -X_1-X_2-X_3 \}.
\end{gather*}
The tropical hypersurface $V(\operatorname{trop}(G))$ is shown in Fig.~\ref{fg:ex2dim}. Let $\rho$ denote the $2$-cell of $V(\operatorname{trop}(G))$ contained in $X_1=1$ and $\nu_1$, $\nu_2$, $\nu_3$, $\mu_1$, $\mu_2$, $\mu_3$ denote faces of~$\rho$ as shown in Fig.~\ref{fg:ex2dim} ($\rho$ denotes the $2$-cell colored in light gray). Fig.~\ref{fg:ex2dim2} shows the intersections of the hyperplane $X_1=1$ and regions $\widehat{D}_{\nu_i}$, $\widehat{D}_{\mu_i}$ $(i=1,2,3)$ and $\widehat{D}_{\rho}$ def\/ined in Def\/inition~\ref{df:nbd}.

\begin{figure}[t]\centering
\begin{minipage}[t]{0.49\hsize}\centering
\includegraphics[width=30mm]{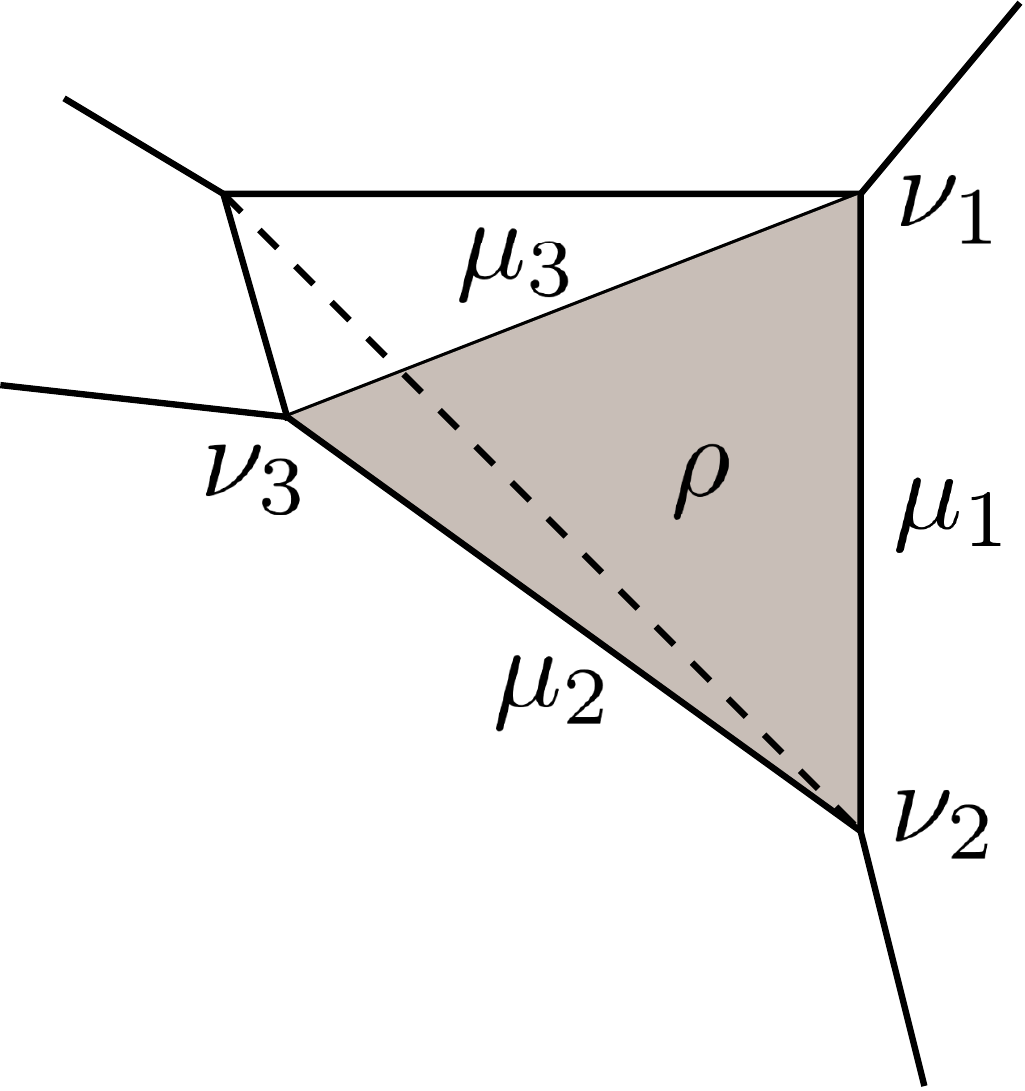}
\caption{The tropical hypersurface $V(\operatorname{trop}(G))$.}\label{fg:ex2dim}
\end{minipage}\quad
\begin{minipage}[t]{0.45\hsize}\centering
\includegraphics[width=35mm]{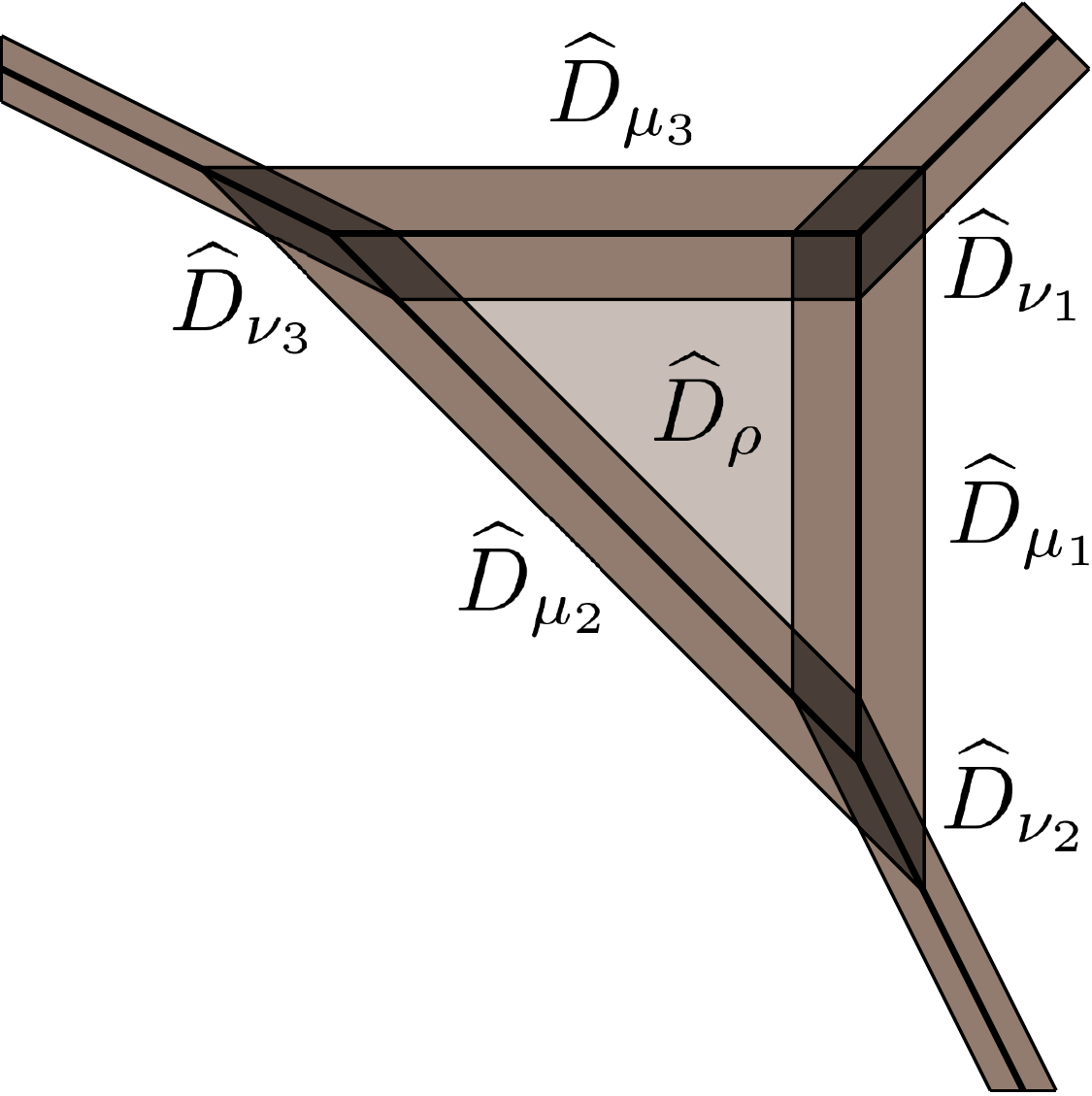}
\caption{The intersections of the hyperplane $X_1=1$ and regions $\widehat{D}_{\nu_i}$, $\widehat{D}_{\mu_i}$, $\widehat{D}_{\rho}$.}\label{fg:ex2dim2}
\end{minipage}
\end{figure}

The set $A_{\nu_1}$ is given by $\{(0,0,0), (1,0,0), (0,1,0), (0,0,1) \}$. We set $\tilde{x}_i:= x_i/ q=q^{-1}x_i$ and $\widetilde{X}_i:=-1+X_i$ for $i=1, 2, 3$. Then the sets of functions $(\tilde{x}_1, \tilde{x}_2, \tilde{x}_3)$ and $(\widetilde{X}_1, \widetilde{X}_2, \widetilde{X}_3)$ form standard coordinates with respect to $\nu_1$ on $O_{\{0\}}({\mathbb C})$ and $O_{\{0\}}({\mathbb T})$. On the other hand, we have $A_{\mu_1}=\{(0,0,0), (1,0,0), (0,1,0) \}$, $A_{\mu_3}=\{(0,0,0), (1,0,0), (0,0,1) \}$ and $A_{\rho}=\{(0,0,0), (1,0,0) \}$. Hen\-ce, the sets of functions $(\tilde{x}_1, \tilde{x}_2, \tilde{x}_3)$ and $(\widetilde{X}_1, \widetilde{X}_2, \widetilde{X}_3)$ also form standard coordinates with respect to $\mu_1$, $\mu_3$ and~$\rho$.

\begin{enumerate}\itemsep=0pt
\item Tropical localization $W_R$ coincides with the following:
 \begin{enumerate}\itemsep=0pt
 \item On $D_{\nu_1} \colon$ the $2$-dimensional tropically localized hyperplane in $(\tilde{x}_1, \tilde{x}_2, \tilde{x}_3)$.
 \item On $D_{\mu_1} \colon$ the direct product of the $1$-dimensional tropically localized hyperplane in $(\tilde{x}_1, \tilde{x}_2)$ and the $1$-dimensional cylinder in~$\tilde{x}_3$.
 \item On $D_{\mu_3} \colon$ the direct product of the $1$-dimensional tropically localized hyperplane in $(\tilde{x}_1, \tilde{x}_3)$ and the $1$-dimensional cylinder in~$\tilde{x}_2$.
 \item On $D_{\rho} \colon$ the direct product of $1$-dimensional cylinders in $\tilde{x}_2$ and~$\tilde{x}_3$.
 \end{enumerate}

\item Let $\phi \colon \operatorname{Log}_R(W_R) \to V(\operatorname{trop}(G))$ be a map satisfying the condition $(\ast)$ in Proposition~\ref{pr:phi}.
We set the map $\phi$ so that the restriction of $\phi$ to $\widehat{D}_{\rho'}$ gives a surjection to~$\rho'$ for any cell $\rho' \in P$.

\item Monodromy transformation $\psi$ is given as follows:
 \begin{enumerate}\itemsep=0pt
 \item On $D_{\nu_1} \colon$ the identity map.
 \item On $D_{\mu_1} \colon$ the map which is identical in the component of the $1$-dimensional tropically localized hyperplane in $(\tilde{x}_1, \tilde{x}_2)$ and coincides with the composition of four times of the Dehn twists in the component of the cylinder in $\tilde{x}_3$.
 \item On $D_{\mu_3} \colon$ the map which is identical in the component of the $1$-dimensional tropically localized hyperplane in $(\tilde{x}_1, \tilde{x}_3)$ and coincides with the composition of four times of the Dehn twists in the component of the cylinder in $\tilde{x}_2$.
 \item On $D_{\rho} \colon$ the map which coincides with the composition of four times of the Dehn twists in both components of the cylinders in $\tilde{x}_2$ and $\tilde{x}_3$.
 \end{enumerate}
 Since the restriction of the map $\phi$ to $\widehat{D}_{\nu_1}$ is the constant map to $\nu_1$ and $\widetilde{X}_i(\nu_1)=0$ for $i=1,2,3$, it follows from Theorem~\ref{th:main} that the restriction of $\psi$ to $D_{\nu_1}$ coincides with the identity map. Since the restriction of the map~$\phi$ to~$D_{\mu_1}$ is a surjection to $\mu_1$ and $\widetilde{X}_3(\nu_1)=0, \widetilde{X}_3(\nu_2)=-4$, we can see from Theorem~\ref{th:main} that the restriction of $\psi$ to $D_{\mu_1}$ coincides with the composition of four times of the Dehn twists in the component of the cylinder in~$\tilde{x}_3$.
 Similarly, it turns out that the restriction of $\psi$ to $D_{\mu_3}$ coincides with the composition of four times of the Dehn twists in the component of the cylinder in $\tilde{x}_2$. On~$D_\rho$, we can also see from Theorem~\ref{th:main} that the map $\psi$ coincides with the composition of four times of the Dehn twists in both components of the cylinders in~$\tilde{x}_2$ and~$\tilde{x}_3$.
\end{enumerate}

\section{Relation to Zharkov's work}\label{sc:7}

\begin{Definition}
A convex lattice polytope $\Delta \subset M_{\mathbb R}$ is {\it smooth} if for each vertex $v$ of $\Delta$, there exists a ${\mathbb Z}$-basis $z_1, \dots, z_{n+1}$ of $M$ such that ${\mathbb R}^{\geq0}(\Delta-v)={\mathbb R}^{\geq 0}z_1+ \cdots +{\mathbb R}^{\geq 0}z_{n+1}$.
\end{Definition}

\begin{Definition}
Let $\Delta \subset M_{\mathbb R}$ be a convex lattice polytope.
We def\/ine the {\it polar polytope} $\Delta^\ast \subset N_{\mathbb R}$ of $\Delta$ by
\begin{gather*}
\Delta^\ast:=\{ n \in N_{\mathbb R} \,|\, \langle m ,n \rangle \geq -1\ \mathrm{for\ all\ } m \in \Delta \}.
\end{gather*}
The convex lattice polytope $\Delta$ is called {\it reflexive} if it contains the origin $0 \in M$ as its interior point and the polar polytope $\Delta^\ast$ is also a lattice polytope in $N_{\mathbb R}$.
\end{Definition}

Let $\Delta$ be a smooth and ref\/lexive polytope in $M_{\mathbb R}$ and $B$ be a subset of $\Delta \cap M$ containing~$0$ and all vertices of~$\Delta$.
Let further $T$ be a coherent triangulation of $(\Delta, B)$. We assume that~$T$ is central, i.e., every maximal-dimensional simplex in $T$ has the origin~$0 \in M$ as it's vertex. Let $\lambda \colon B \to {\mathbb Z}$ be an integral vector which is in the interior of the secondary cone (see \cite[Chap\-ter~7, Def\/inition~1.4]{MR2394437}) corresponding to $T$. We consider the function $f_q$ def\/ined by
\begin{gather*}
f_q(x):=q^{\lambda(0)}-\sum_{i \in B \setminus \{0\}} q^{\lambda(i)}x^{i}.
\end{gather*}
where $q \in S_R^1:=\{ z \in {\mathbb C} \,|\, |z|=R \}$ for a suf\/f\/iciently large $R \in {\mathbb R}^{>0}$. Let $X_\Delta$ be the toric manifold whose moment polytope is~$\Delta$ and~$V_q$ be the hypersurface in $X_\Delta$ def\/ined by~$f_q$. In this setting, Zharkov constructed the monodromy transformation of $\{V_q\}_{q \in S_R^1}$ as follows:
\begin{enumerate}\itemsep=0pt
\item[(i)] Let $\mu_R \colon X_\Delta \to \Delta$ be the weighted moment map def\/ined by
\begin{gather*}
\mu_R(x):=\frac{\sum\limits_{m \in B}R^{\lambda(m)} |x^m|m}{\sum\limits_{m \in B}R^{\lambda(m)} |x^m|}.
\end{gather*}
There exists a small neighborhood $U \subset \Delta$ of the origin $0 \in \Delta$ such that $\mu_R(V_q) \subset \Delta \setminus U$ for any $q \in S_R^1$.
We set $\Delta^\circ := \Delta \setminus U$. He constructs two families of regions $\{U_\tau \}_{\tau \in \partial T}$ and $\{\widetilde{U}_\tau \}_{\tau \in \partial T}$ in $\Delta^\circ$. For instance, in the case where
\begin{gather}\label{eq:ex}
f_q:=q-\big( x+xy+y+x^{-1}+x^{-1}y^{-1}+y^{-1} \big)
\end{gather}
and the triangulation $T$ is given as shown in Fig.~\ref{fg:newdiv}, the families of regions $\{U_\tau \}_{\rho \in \partial T}$ and $\{\widetilde{U}_\tau \}_{\tau \in \partial T}$ are as shown in Figs.~\ref{fg:U} and~\ref{fg:tU}, respectively. $v_i$ and $w_i$ $(i=1, \dots, 6)$ denote vertices and edges of~$\Delta$ respectively as shown in Fig.~\ref{fg:newdiv}. $U_{v_i}$, $\tilde{U}_{v_i}$ denote the regions colored in light gray and $U_{w_i}$, $\tilde{U}_{w_i}$ denote the regions colored in dark gray as shown in Figs.~\ref{fg:U} and~\ref{fg:tU}. We omit their construction here and refer the reader to \cite[Section~3]{MR1738179} about how to construct them.

\begin{figure}[t]\centering
\includegraphics[scale=0.40]{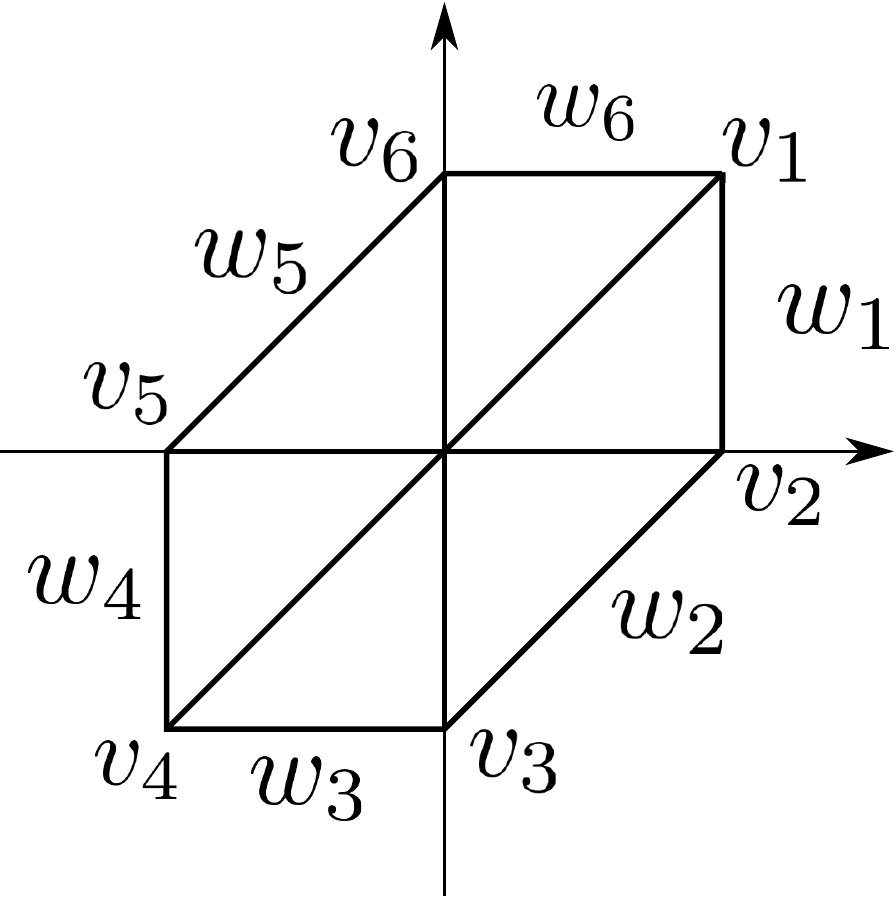}
\caption{The triangulation $T$ given by \eqref{eq:ex}.}\label{fg:newdiv}
\end{figure}

\begin{figure}[t]\centering
\begin{minipage}[t]{0.48\hsize}\centering
\includegraphics[width=45mm]{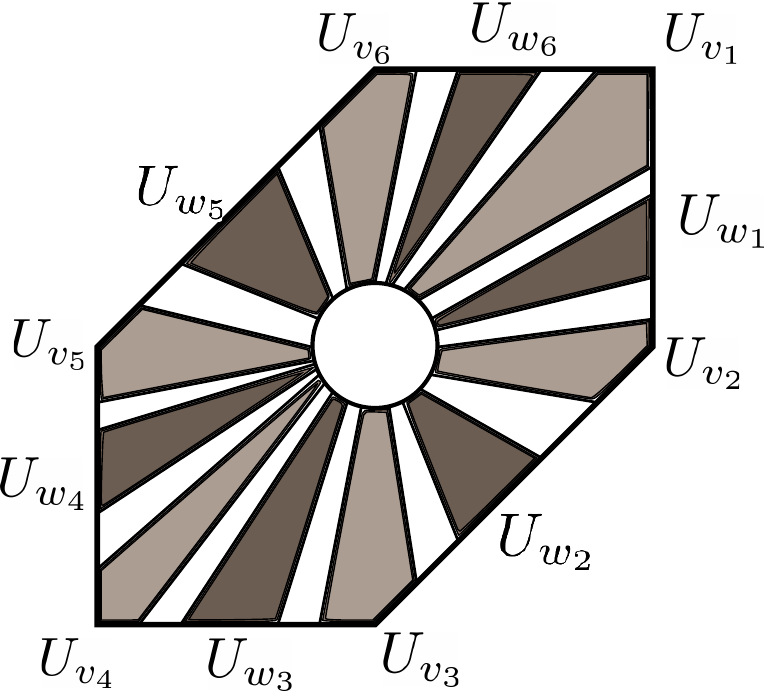}
\caption{Regions $\{U_\tau\}_{\tau \in \partial T}$.}\label{fg:U}
\end{minipage}\quad
\begin{minipage}[t]{0.48\hsize}\centering
\includegraphics[width=45mm]{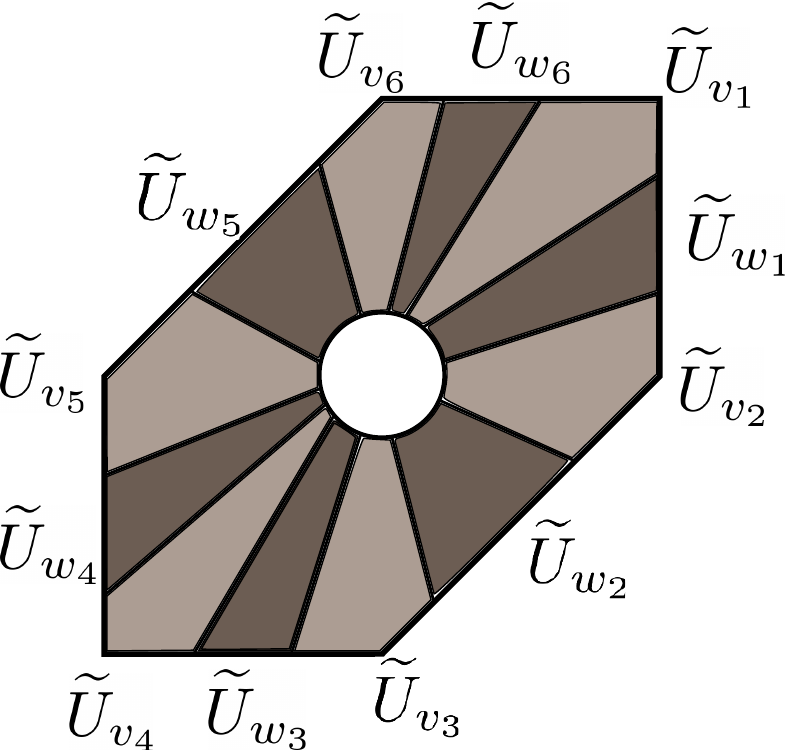}
\caption{Regions $\{\widetilde{U}_\tau\}_{\tau \in \partial T}$.}\label{fg:tU}
\end{minipage}
\end{figure}

\item[(ii)]
He sets bump functions $b_m \colon \Delta^\circ \to [0,1]\ (m \in B \setminus \{0\})$ so that the function $\tilde{f}_q \colon ({\mathbb C}^\ast)^{n+1} \to {\mathbb C}$ def\/ined by
\begin{gather*}
\tilde{f}_q(x):=q^{\lambda(0)}-\sum_{m \in B \setminus \{0\}} (b_m \circ \mu_R) (x)q^{\lambda(m)}x^{m}
\end{gather*}
coincides with
\begin{gather*}
\begin{split}
& q^{\lambda(0)}-\sum_{m \in \tau \cap B} q^{\lambda(m)}x^{m} \qquad \mathrm{on} \quad \mu_R^{-1}(U_\tau) \cap ({\mathbb C}^\ast)^{n+1}, \\
& q^{\lambda(0)}-\sum_{m \in \tau \cap B} (b_m \circ \mu_R) (x)q^{\lambda(m)}x^{m} \qquad \mathrm{on} \quad \mu_R^{-1}(\widetilde{U}_\tau) \cap ({\mathbb C}^\ast)^{n+1},
\end{split}
\end{gather*}
for any $\tau \in \partial T$. Let $W_q$ denote the submanifold in $X_\Delta$ def\/ined by $\tilde{f}_q(x)=0$. We can see from the def\/inition of the weighted moment map $\mu_R$ that if $\mu_R(x) \in \widetilde{U}_\tau$, the dominant part of $f_q$ at $x$ are $q^{\lambda(0)}-\sum\limits_{m \in \tau \cap B} q^{\lambda(m)}x^{m}$. Since orders of terms cut of\/f by bump functions $\{b_m\}_{m \in B \setminus \{0\}}$ are lower, the submanifold $W_q$ is dif\/feomorphic to $V_q$.
\item[(iii)]
He def\/ines the family of subsets $\{ \Delta_\gamma^\vee \subset N_{\mathbb R}\}_{\gamma \in [0,1]}$ by
\begin{gather*}
\Delta_\gamma^\vee := \big\{ n \in N_{\mathbb R} \,|\, {-}\langle m, n \rangle \geq \gamma (\lambda(m) - \lambda(0)) \ \mathrm{for\ any\ vertex} \ m \ \mathrm{in} \ T \big\}.
\end{gather*}
For any $\gamma >0$, the set $\Delta_\gamma^\vee$ is a convex polytope with a nonempty interior. The set $\Delta_\gamma^\vee$ in the case where $f_q$ is given by~\eqref{eq:ex} is shown in Fig.~\ref{fg:dupoly}.

\begin{figure}[t]\centering
\includegraphics[scale=0.45]{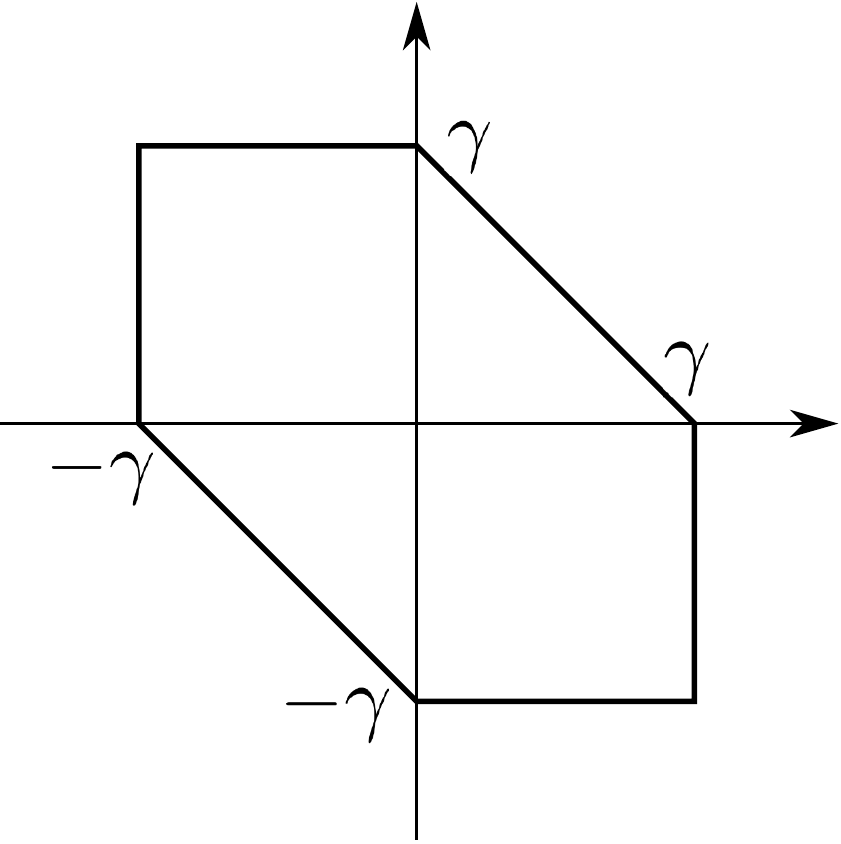}
\caption{The convex polytope $\Delta_\gamma^\vee$ in the case where $f_q$ is given by~\eqref{eq:ex}.}\label{fg:dupoly}
\end{figure}

The region surrounded by the center part of the tropical hypersurface coincides with
\begin{gather*}
\big\{ n \in N_{\mathbb R} \,|\, \lambda(0) \geq \langle m, n \rangle + \lambda(m) \ \mathrm{for\ any\ vertex} \ m \ \mathrm{in} \ T \big\}.
\end{gather*}
Hence when we set $\gamma=1$, the boundary of the convex polytope $\Delta_{\gamma=1}^\vee$ coincides with the center part of the tropical hypersurface.
For each $k$-dimensional simplex $\tau \in \partial T$, we def\/ine an $(n-k)$-dimensional face $\tau^\vee$ of $\Delta_\gamma^\vee$ by
\begin{gather*}
\tau^\vee:=\big\{ n \in \Delta_\gamma^\vee \,|\, {-}\langle n, m \rangle = \gamma (\lambda(m) - \lambda(0)) \ \mathrm{for\ any\ vertex} \ m \ \mathrm{in} \ \tau \big\}.
\end{gather*}
There is a bijective correspondence between simplices in~$\partial T$ and faces of $\Delta_\gamma^\vee$ given by $\tau \leftrightarrow \tau^\vee$.
Then he constructs a family of maps $\{v_\gamma \colon \Delta^\circ \to \partial \Delta_\gamma^\vee \}_{\gamma \in [0,1]}$ which depends on~$\gamma$ smoothly and satisf\/ies $v_\gamma(\widetilde{U}_\tau) \subset \tau^\vee$ for any $\tau \in \partial T$.

\item[(iv)]
Let $e_i:=(0,\dots,0,\overset{i}{\check{1}},0,\dots,0) \in M$ $(i=1, \dots, n+1)$ be the unit vector
and $\psi_{i,\gamma} \colon X_\Delta \to {\mathbb C}$ $(i=1, \dots, n+1)$ be the function def\/ined by
\begin{gather*}
\psi_{i,\gamma}:=\exp \big( 2 \pi \sqrt{-1} \langle (v_\gamma \circ \mu_R )(x), e_i \rangle \big).
\end{gather*}
He def\/ines a family of dif\/feomorphisms $\{ D_\gamma \colon X_\Delta \to X_\Delta \}_{\gamma \in [0,1]}$ by
\begin{gather}\label{eq:d1}
(x_{1}, \dots, x_{n+1}) \to (\psi_{1,\gamma} x_1, \dots, \psi_{n+1,\gamma} x_{n+1}).
\end{gather}
For any element $x \in W_R \cap \mu_R^{-1}(\widetilde{U}_\tau)$, we have $\mu_R(D_\gamma(x))=\mu_R(x) \in \widetilde{U}_\tau$ and
\begin{gather*}
\tilde{f}_q(D_\gamma(x)) =q^{\lambda(0)}-\sum_{m \in \tau \cap B}(b_m \circ \mu_R)(x)q^{\lambda(m)}x^m \exp \big( 2 \pi \sqrt{-1} \langle v_\gamma (\mu_R(x)), m\rangle \big) \\
\hphantom{\tilde{f}_q(D_\gamma(x))}{}
=q^{\lambda(0)}-\sum_{m \in \tau \cap B}(b_m \circ \mu_R)(x)q^{\lambda(m)}x^m \exp \big( {-}2 \pi \sqrt{-1} \gamma ( \lambda(m) - \lambda(0)) \big)\\
\hphantom{\tilde{f}_q(D_\gamma(x))}{}
 =\exp(2 \pi \sqrt{-1} \gamma \lambda(0)) \left\{ R^{\lambda(0)}-\sum_{m \in \tau \cap B}(b_m \circ \mu_R)(x)R^{\lambda(m)}{x}^m \right\}=0 .
\end{gather*}
Hence, the family of maps $\{D_\gamma \colon X_\Delta \to X_\Delta \}_{\gamma \in [0,1]}$ induces the monodromy transformation $\big\{ D_\gamma \colon W_R \to W_{q=R\exp (2 \pi \sqrt{-1} \gamma)} \big\}_{\gamma \in [0,1]}$.
\end{enumerate}

\begin{figure}[t!]\centering
\includegraphics[scale=0.3]{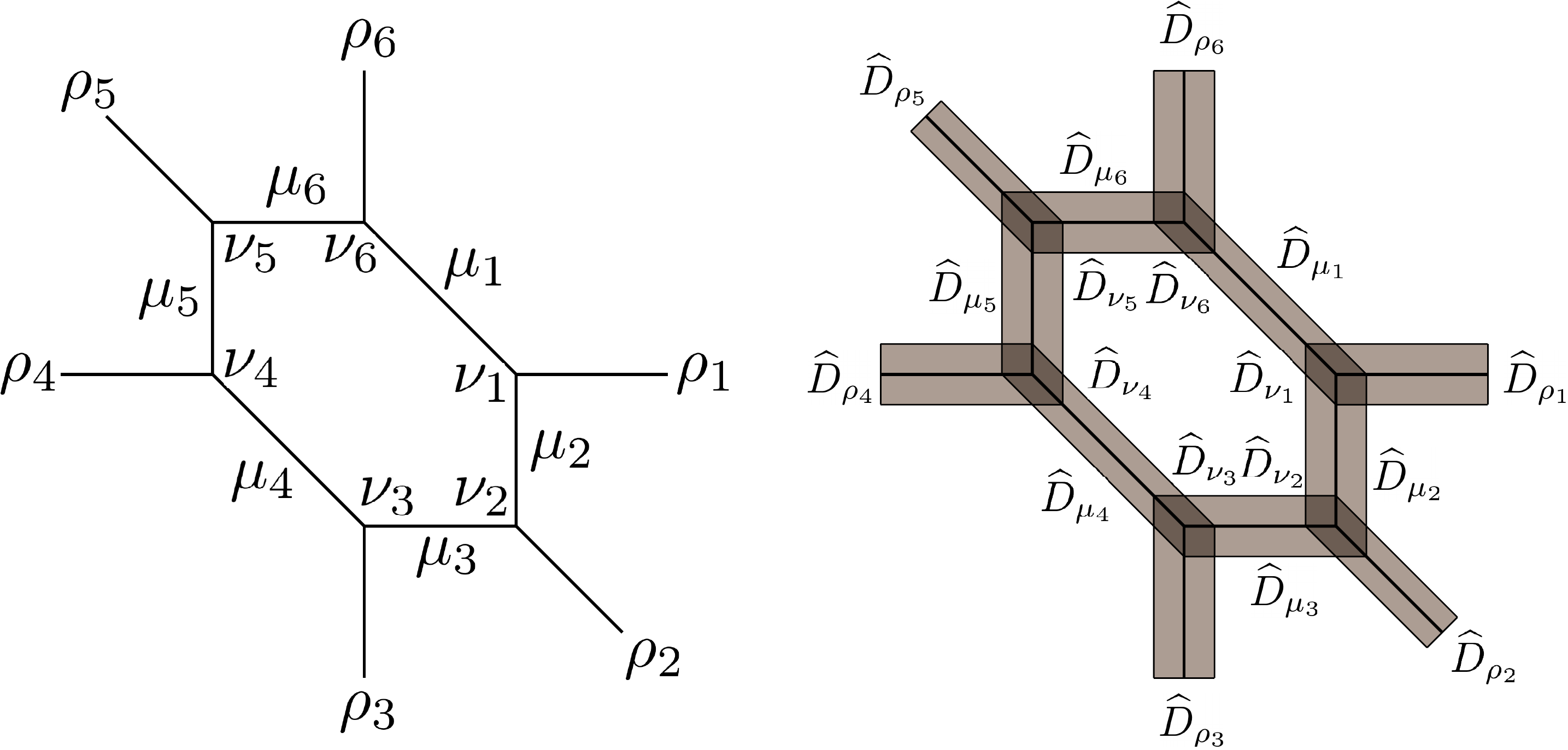}
\caption{The tropical hypersurface and $\{\widehat{D}_\mu\}_\mu$ in the case where~$f_q$ is given by~\eqref{eq:ex}.}\label{fg:our}
\end{figure}

As explained in~(ii), Zharkov also localized the hypersurface $V_q$ to construct the monodromy transformation. He used the weighted moment map while we used the tropicalization. The regions $\{\widetilde{U}_\tau\}_\tau$ are similar to $\{\widehat{D}_\mu\}_\mu$ constructed in Def\/inition~\ref{df:nbd}. Moreover, terms which we cut of\/f at each region are also the same. The tropical hypersurface and the family of regions $\{\widehat{D}_\mu\}_\mu$ are shown in Fig.~\ref{fg:our} in the case where~$f_q$ is given by~\eqref{eq:ex}. The region $\widetilde{U}_{v_i}$ corresponds to $\widehat{D}_{\mu_i}$ and $\widetilde{U}_{w_i}$ corresponds to $\widehat{D}_{\nu_i}$ $(i=1,\dots, 6)$, respectively. For instance, on both $\widetilde{U}_{v_2}$ and $\widehat{D}_{\mu_2}$, the dominant terms are $q$ and $x$. On both $\widetilde{U}_{w_1}$ and $D_{\nu_1}$, the dominant terms are~$q$, $x$, $xy$, and so on.
Note that regions at which the term $q^{\lambda(0)}$ is not dominant in our construction are included in other regions in Zharkov's construction. For instance, in the case~$f_q$ is given by~\eqref{eq:ex}, the region corresponding to $\widehat{D}_{\rho_i}$ is included in $\widetilde{U}_{w_i}$ for $i=1,\dots,6$.
This is the only major dif\/ferences in the localization and the resulting manifolds~$W_q$ are similar to each other.

His construction of the monodromy transformation is also similar to ours.
We can construct the family of maps $\{ v_\gamma \colon \Delta^\circ \to \partial \Delta_\gamma^\vee \}_{\gamma \in [0,1]}$ as follows.
First, we construct $v_{\gamma=1}$ satisfying $v_\gamma(\widetilde{U}_\rho) \subset \rho^\vee$ for any $\rho \in \partial T$. We set
\begin{gather*}
S_\gamma \colon \ M_{{\mathbb R}} \to M_{{\mathbb R}}, \qquad (X_1, \dots, X_{n+1}) \to (\gamma X_1, \dots, \gamma X_{n+1}).
\end{gather*}
for each $\gamma \in [0,1]$. Then the map $v_\gamma:=S_\gamma \circ v_{\gamma=1}$ satisf\/ies requested conditions. The map $\phi \colon \operatorname{Log}_R(W_R) \to V(\operatorname{trop}(F))$ in Proposition~\ref{pr:phi} plays the same role as~$v_{\gamma=1}$. Moreover, the monodromy transformation given by~\eqref{eq:c4} in our construction coincides with~\eqref{eq:d1}. It can be said that our construction is a natural generalization of Zharkov's construction.

\subsection*{Acknowledgements}
The author would like to express his gratitude to Kazushi Ueda for encouragement and helpful advices. The author thanks to Tatsuki Kuwagaki for explaining the context of the paper~\cite{DKK}. The author also thanks the anonymous referees for reading this paper carefully and giving many helpful comments. This research is supported by the Program for Leading Graduate Schools, MEXT, Japan.

\pdfbookmark[1]{References}{ref}
\LastPageEnding


\begin{thebibliography}{99}
\footnotesize\itemsep=0pt

\bibitem{MR2240909}
Abouzaid M., Homogeneous coordinate rings and mirror symmetry for toric
 varieties, \href{http://dx.doi.org/10.2140/gt.2006.10.1097}{\textit{Geom. Topol.}} \textbf{10} (2006), 1097--1157,
 \href{http://arxiv.org/abs/math.SG/0511644}{math.SG/0511644}.

\bibitem{DKK}
Diemer C., Katzarkov L., Kerr G., Symplectomorphism group relations and
 degenerations of {L}andau--{G}inzburg models, \href{http://arxiv.org/abs/1204.2233}{arXiv:1204.2233}.

\bibitem{MR2394437}
Gelfand I.M., Kapranov M.M., Zelevinsky A.V., Discriminants, resultants and
 multidimensional determinants, \textit{Modern Birkh\"auser Classics}, Birkh\"auser
 Boston, Inc., Boston, MA, 2008.

\bibitem{IwaoLecture2010}
Iwao S., Complex integration vs tropical integration, {L}ecture at The
 {M}athematical {S}ociety of {J}apan {A}utum {M}eeting, 2010, available at \url{http://mathsoc.jp/videos/2010shuuki.html}.

\bibitem{KajiwaraPre}
Kajiwara T., Tropical toric varieties, {P}reprint, Tohoku University, 2007.

\bibitem{MR2428356}
Kajiwara T., Tropical toric geometry, in Toric Topology, \href{http://dx.doi.org/10.1090/conm/460/09018}{\textit{Contemp.
 Math.}}, Vol.~460, Amer. Math. Soc., Providence, RI, 2008, 197--207.

\bibitem{MR3287221}
Maclagan D., Sturmfels B., Introduction to tropical geometry, \textit{Graduate
 Studies in Mathematics}, Vol.~161, Amer. Math. Soc., Providence, RI, 2015.

\bibitem{MR2079993}
Mikhalkin G., Decomposition into pairs-of-pants for complex algebraic
 hypersurfaces, \href{http://dx.doi.org/10.1016/j.top.2003.11.006}{\textit{Topology}} \textbf{43} (2004), 1035--1065,
 \href{http://arxiv.org/abs/math.GT/0205011}{math.GT/0205011}.

\bibitem{ISSN:1401-5617}
Rullg{\aa}rd H., Polynomial amoebas and convexity, {P}reprint, Stockholm
 University, 2001, available at \url{http://www2.math.su.se/reports/2001/8/2001-8.pdf}.

\bibitem{MR1738179}
Zharkov I., Torus f\/ibrations of {C}alabi--{Y}au hypersurfaces in toric
 varieties, \href{http://dx.doi.org/10.1215/S0012-7094-00-10124-X}{\textit{Duke Math.~J.}} \textbf{101} (2000), 237--257,
 \href{http://arxiv.org/abs/math.AG/9806091}{math.AG/9806091}.

\end{thebibliography}
\end{document}